\documentclass[10pt]{article}
\usepackage{amsmath}
\usepackage{amsfonts}
\usepackage{amssymb}
\usepackage{amsthm}

\begin{document}

\newcommand{\C}{{\mathbb{C}}}
\newcommand{\R}{{\mathbb{R}}}
\newcommand{\Z}{{\mathbb{Z}}}
\renewcommand{\l}{\left}
\renewcommand{\r}{\right}
\newcommand{\lap}{\bigtriangleup}
\newcommand{\cptsubset}{\subset\subset}
\newcommand{\boxb}{\square_b}
\newcommand{\dbarb}{\overline{\partial}_b}
\newcommand{\Lbar}{\overline{L}}
\newcommand{\sL}{\mathcal{L}}
\newcommand{\sLt}{\widetilde{\sL}}
\newcommand{\LpN}[3]{\l\| #2 \r\|_{L^{#1}\l( #3 \r)}}
\newcommand{\LpOpN}[3]{\l\| #2 \r\|_{L^{#1}\l(M\r)\circlearrowleft}}
\newcommand{\LpsN}[4]{\l\| #3 \r\|_{L_{#2}^{#1}\l( #4 \r)}}
\newcommand{\LtaN}[2]{\l\| #2\r\|_{L^2_{#1}\l( \R\r)}}
\newcommand{\LtaNloc}[2]{\l\| #2\r\|_{L^2_{#1,\mathrm{sloc}}}}
\newcommand{\rp}[2]{\rho\l( #1,#2 \r)}
\newcommand{\dip}[2]{d\l( #1,#2 \r)}
\newcommand{\boxbt}{\widetilde{\square}_b}
\newcommand{\bt}{\widetilde{b}}
\newcommand{\psit}{\widetilde{\psi}}
\newcommand{\Lambdat}{\widetilde{\Lambda}}
\newcommand{\supp}{\mathrm{supp}}
\renewcommand{\Re}{\mathrm{Re}}
\renewcommand{\Im}{\mathrm{Im}}
\newcommand{\ip}[2]{\l< #1, #2 \r>}
\newcommand{\sM}{\mathcal{M}}
\newcommand{\loc}{\mathrm{loc}}

\newtheorem{thm}{Theorem}[section]
\newtheorem{cor}[thm]{Corollary}
\newtheorem{prop}[thm]{Proposition}
\newtheorem{lemma}[thm]{Lemma}
\newtheorem{conj}[thm]{Conjecture}

\theoremstyle{remark}
\newtheorem{rmk}[thm]{Remark}

\theoremstyle{definition}
\newtheorem{defn}[thm]{Definition}

\title{The $\boxb$ Heat Equation and Multipliers via the Wave Equation}
\author{Brian Street}
\date{}

\maketitle
\begin{abstract}
Recently, Nagel and Stein studied the $\boxb$-heat equation,
where $\boxb$ is the Kohn Laplacian on the boundary of a
weakly-pseudoconvex domain of finite type in $\C^2$.  They showed
that the Schwartz kernel of $e^{-t\boxb}$ satisfies good
``off-diagonal'' estimates, while that of $e^{-t\boxb}-\pi$
satisfies good ``on-diagonal'' estimates, where $\pi$ denotes
the Szeg\"o projection.  We offer a simple
proof of these results, which easily generalizes to other,
similar situations.  Our methods involve adapting the well-known
relationship between the heat equation and the finite propagation speed of the wave equation
to this situation.  In addition, we apply these methods
to study multipliers of the form $m\l( \boxb\r)$.  In particular,
we show that $m\l( \boxb\r)$ is an NIS operator, where
$m$ satisfies an appropriate Mihlin-H\"ormander condition.

\end{abstract}

\section{Introduction}
In \cite{NagelSteinTheBoxbHeatEquationOnPseudoconvexManifoldsOfFiniteType},
Nagel and Stein study the heat operator $e^{-s\boxb}$, where $\boxb$
is the Kohn Laplacian (acting on functions) on the boundary $M$,
of a weakly pseudoconvex domain of finite type in $\C^2$ (or with
$\boxb$ on a polynomial model domain in $\C^2$).  Let $\pi$
be the Szeg\"o projection, $e^{-s\boxbt}=\l(1-\pi\r) e^{-s\boxb}$, and
for any operator $T$, let $K_T$ denote the Schwartz kernel of $T$.
The bounds in \cite{NagelSteinTheBoxbHeatEquationOnPseudoconvexManifoldsOfFiniteType} were in terms of a Carnot-Carath\'eodory distance $\rho$ on $M$
(see Section \ref{SectionGeomOfM}).  In \cite{NagelSteinTheBoxbHeatEquationOnPseudoconvexManifoldsOfFiniteType} it is shown that:
\begin{equation}\label{EqnIntroOffDiag}
\l|K_{e^{-s\boxb}}\l(x,y\r) \r|\lesssim \frac{1}{V\l(x,\rp{x}{y}\r)} \l( \frac{s^N}{s^N+\rp{x}{y}^{2N}} \r)
\end{equation}
for every $N>0$, and,
\begin{equation}\label{EqnIntroOnDiag}
\l|K_{e^{-s\boxbt}}\l(x,y\r) \r| \lesssim \frac{1}{V\l(x,\rp{x}{y}\vee \sqrt{s}\r)}
\end{equation}
with appropriate estimates for the derivatives in each variable as well (see Theorem \ref{ThmMainBounds}).  Here, $V\l(x,\delta\r)$ denotes the volume
of the ball of radius $\delta$ in the $\rho$ metric, centered at $x$.
In an unpublished result of Nagel and M\"uller, using the methods of \cite{JerisonSanchezCalleEstimatesForTheHeatKernelForASumeOfSquares}
the bounds in (\ref{EqnIntroOffDiag}) are improved to:
\begin{equation}\label{EqnIntroOffDiagGaus}
\l|K_{e^{-s\boxb}}\l(x,y\r) \r|\lesssim \frac{1}{V\l(x,\rp{x}{y}\r)} e^{-c\frac{\rp{x}{y}^2}{s}}
\end{equation}
for some $c>0$.  The main insight of \cite{NagelSteinTheBoxbHeatEquationOnPseudoconvexManifoldsOfFiniteType} was that one needs
to prove off diagonal estimates (ie, (\ref{EqnIntroOffDiag}),(\ref{EqnIntroOffDiagGaus}))
for $e^{-s\boxb}$ and on diagonal estimates (ie, (\ref{EqnIntroOnDiag}))
for $e^{-s\boxbt}$.  The main goal of this paper is to reprove these
results, keeping this insight in mind, using well-known methods
for the classical heat equation as can be found, for instance,
in \cite{SikoraRieszTransformGaussianBoundsAndTheMethodOfWaveEquation}.

The novelty of our approach is that we shall use only two estimates
specific to $\boxb$.  Namely,
\begin{itemize}
\item There is a relative fundamental solution $\boxbt^{-1}$ (ie, $\boxb\boxbt^{-1}=\boxbt^{-1}\boxb = 1-\pi$, $\pi\boxbt^{-1}=0=\boxbt^{-1}\pi$) 
which is an NIS operator of order $2$ (see Definition \ref{DefnNisOps}, and Theorem \ref{ThmScaledEst} for the related estimate).
\item The $\boxb$ wave equation has finite propagation speed (see Theorem \ref{ThmFiniteProp}).
\end{itemize}
and the rest of the proofs follows completely formally, from a modified version
of the proofs in \cite{SikoraRieszTransformGaussianBoundsAndTheMethodOfWaveEquation}.  In particular, we do not need any of the new bounds that
were used in \cite{NagelSteinTheBoxbHeatEquationOnPseudoconvexManifoldsOfFiniteType}.

That $\boxbt^{-1}$ is an NIS operator of order $2$ is well known 
(see \cite{NagelRosaySteinWaingerEstimatesForTheBergmanAndSzegoKernels,ChangNagelSteinEstimatesForTheDbarNeumannProblemInPseudoconvexDomainsOfFiniteType}), 
while
the finite propagation speed is a result of Melrose \cite{MelrosePropagationForTheWaveGroupOfAPositiveSubellipticSecondOrderDifferentialOperator}.
Because of this, essentially no new estimates need to be
proven to achieve the main results of this paper:  all of the work is completely formal use
of the spectral theorem.

In fact, one may consider the methods of this paper as a (quite simple) generalization
of the methods in \cite{SikoraRieszTransformGaussianBoundsAndTheMethodOfWaveEquation} where the heat equation $e^{-t\sL}$ is studied for some
positive semi-definite operator $\sL$, whose wave equation
has finite speed of propagation.  The methods in this paper
allow one to consider the case when the $L^2$ kernel of $\sL$
is non-trivial,\footnote{Here, ``trivial'' does not necessarily mean $0$ dimensional.  It could, for instance when working on a compact manifold, mean a finite dimensional space of smooth functions.}
and the Schwartz kernel of the orthogonal projection onto the 
$L^2$ kernel of $\sL$ satisfies appropriate estimates (see Example
\ref{ExampleSik}).

After we study the $\boxb$ heat equation, in Section \ref{SectionMultipliers}
we turn to studying multipliers $m\l( \boxb\r)$.  We show (using
essentially the same methods that we use for the heat equation)
that $m\l(\boxb\r)$ is an NIS operator of order $0$, provided $m$
satisfies an appropriate Mihlin-H\"ormander condition (see
Theorem \ref{ThmMultipliers}).  In addition, we prove $m\l( \boxb\r)$
is bounded on $L^p$ for $m$s that satisfy only a finite
level of smoothness (see Theorem \ref{ThmMultipliersLessSmooth}).

Finally, since we use only the two above basic assumptions
on $\boxb$, in Section \ref{SectionOtherExamples} we present
a few other examples where the same methods yield analogous results.
We hope this will convince the reader that these methods
are easily adapted to other situations.

All of the results in this paper are well known when $\boxb$ is
replaced by, for instance, the sublaplacian on a compact manifold,
defined in terms of vector fields $X_1,\ldots, X_m$ which
satisfy H\"ormander's condition, or the sublaplacian on
a stratified group.  See Examples \ref{ExampleCptMfld} and \ref{ExampleQuasiHomog}, and the references there.

\section{Setup, Notation, and Statement of Results}
Despite the fact that the methods of this paper work in more general
situations (see Section \ref{SectionOtherExamples}; in particular,
Examples \ref{ExampleSik} and \ref{ExamplePseudoConvex})
it seems difficult to devise an appropriate abstract setting
in which the entirety of this paper will go through, without
needlessly complicating matters (in particular,
the obvious generalization (Example \ref{ExampleCptMfld}) doesn't contain Example \ref{ExamplePseudoConvex}).
Because of this, we prove our results, in detail, in the simplest setting (as discussed
in the introduction) and mention a few other settings in which these
methods work, with only minor modifications, in Section \ref{SectionOtherExamples}.
Throughout the paper we will use $A\lesssim B$ to
denote $A\leq CB$ where $C$ is a constant independent of any
relevant parameters.

Let $M\subset \C^2$ be a $C^\infty$ pseudoconvex hypersurface 
and assume that $M$ is the boundary of a bounded
pseudoconvex domain $\Omega\subset \C^2$ of finite type $m$.  Let $\boxb$ denote
the Kohn-Laplacian acting on functions in $C^\infty\l(M\r)$.  Let
$\pi:L^2\l( M\r)\rightarrow L^2\l( M\r)$ be the orthogonal
projection on to the $L^2$ kernel of $\boxb$ (ie, the Szeg\"o
projection).  

We write $L^p_s\l(M\r)$ for the usual (isotropic) $L^p$ Sobolev spaces
of order $s$.  Also, if $\zeta,\zeta'\in C_0^\infty\l( M\r)$ we write
$\zeta\prec \zeta'$ if $\zeta'\equiv 1$ on the support of $\zeta$.
Finally, for an operator $T:C_0^\infty\l( M\r) \rightarrow C_0^\infty\l(M\r)'$, we write $K_T\l(x,y\r)\in C_0^\infty\l( M\times M\r)'$ for
the Schwartz kernel of the operator $T$.

	\subsection{Geometry of M}\label{SectionGeomOfM}
		Choose real vector fields $X_1,X_2$ so that we can identify
$\dbarb f$ with $$\l( X_1+iX_2\r) f$$ by identifying functions and
$(0,1)$ forms, see \cite{NagelSteinTheBoxbHeatEquationOnPseudoconvexManifoldsOfFiniteType} for details.
That $\Omega$ is of finite type $m$ means that $X_1$, $X_2$ and their
commutators up to order $m$ span the tangent space $TM$ at each point (ie,
$X_1$ and $X_2$ satisfy H\"ormander's condition).  There
is a natural metric defined in terms of these vector fields, called
the control metric, or Carnot-Carath\'eodory metric, and is defined by:
\begin{equation*}
\begin{split}
\rp{x}{y} = \inf\bigg\{ T>0 : &\exists \gamma:\l[0,T\r]\rightarrow M, \quad\gamma \: \mathrm{piecewise}\: C^1,\\
& \gamma\l(0\r) =x, \gamma\l( T\r) =y,  \\
& \gamma'\l(t\r)= c_1\l(t\r) X_1 + c_2\l(t\r) X_2\text{ for a.e. $t$,}\\
& \text{and } \l| c_1\l(t\r) \r|^2 + \l|c_2\l(t\r) \r|^2\leq 1 \bigg\}
\end{split}
\end{equation*}
We define $B\l( x, \delta\r)=\l\{ y\in M: \rp{x}{y}<\delta \r\}$ and
let $V\l( x,\delta\r)$ denote the volume of $B\l( x, \delta\r)$.
The following result is contained in \cite{NagelSteinWaingerBallsAndMetricsDefinedByVectorFields}:
\begin{prop}\label{PropHomogType}
There is a $Q$ such that $V\l( x, \gamma \delta\r) \lesssim \gamma^Q V\l( x, \delta\r)$ for all $\gamma\geq 1$.  For the rest of the paper $Q$ will denote
this number.  In addition, there is a $q>0$, and a $\delta_0>0$ such that
for all $\delta\leq \delta_0$ and all $\gamma<1$, we have
$V\l(x,\gamma\delta\r) \lesssim \gamma^q V\l(x,\delta\r)$.
\end{prop}
\begin{rmk}
It is not hard to see, from the results of \cite{NagelSteinWaingerBallsAndMetricsDefinedByVectorFields}, that $Q=m+2$, where $M$ was of finite type $m$.
\end{rmk}

We write $D=\l( X_1,X_2\r)$ and use order multi-index notation:  $D^\alpha$
where $\alpha$ is a sequences of $1$s and $2$s, and $\l|\alpha\r|$ denotes
the length of that sequence.  So that, for instance, $D^{\l( 1,2,1,1\r)}=X_1X_2X_1X_1$ and $\l|\l( 1,2,1,1\r)\r|=4$.

	\subsection{The Operator $\boxb$}
		As in \cite{NagelSteinTheBoxbHeatEquationOnPseudoconvexManifoldsOfFiniteType},
we have identified $\dbarb$ with a linear first order partial differential
operator (by identifying functions and $(0,1)$ forms), and $\dbarb^{*}$ with its
 adjoint (also a linear first
order partial differential operator).  Then, as in \cite{NagelSteinTheBoxbHeatEquationOnPseudoconvexManifoldsOfFiniteType}
we may define $\boxb=\dbarb^{*}\dbarb$, and see that with an appropriate definition of domain, $\boxb$ is a self-adjoint
operator (we refer the reader to \cite{NagelSteinTheBoxbHeatEquationOnPseudoconvexManifoldsOfFiniteType} for the details of the Hilbert space theory).

Since $\boxb$ is a self-adjoint operator, it admits a spectral decomposition
$E\l( \lambda\r)$; so that, in particular, $\pi=E\l( 0\r)$.  Hence, for any bounded Borel measurable function $F:\l[0,\infty\r)\rightarrow \C$, we may define:
$$F\l( \boxb\r) = \int_{\l[ 0,\infty\r)} F\l( \lambda\r) dE\l( \lambda\r)$$
and with an abuse of notation, we define:
$$F\l( \boxbt\r) = \int_{\l( 0,\infty\r)} F\l( \lambda\r) dE\l( \lambda\r)$$
So that
$$F\l( \boxbt\r) = \l( 1-\pi\r) F\l( \boxb\r) = F\l( \boxb\r) - F\l( 0\r) \pi$$

	\subsection{Statement of Results}\label{SectionResults}
		\begin{thm}[\cite{MelrosePropagationForTheWaveGroupOfAPositiveSubellipticSecondOrderDifferentialOperator}]\label{ThmFiniteProp}
There exists a constant $\kappa>0$ such that:
$$\supp\l( K_{\cos\l( t\sqrt{\boxb}\r)}\r)\subseteq \l\{\l( x,y\r): \rp{x}{y}\leq \kappa t \r\}$$
See Section \ref{SectionFiniteSpeed} for a discussion of this result and
for another proof.
\end{thm}
We will fix the constant $\kappa$ as in Theorem \ref{ThmFiniteProp},
and all of our other results will be in terms of this $\kappa$.  Our
main results are now as follows:

As in \cite{NagelSteinTheBoxbHeatEquationOnPseudoconvexManifoldsOfFiniteType},
we study the operators $e^{-t\boxb}$ and $e^{-t\boxbt}$ which satisfy:
\begin{equation}\label{EqnRelBetweenHeats}
e^{-t\boxb}=e^{-t\boxbt}+\pi
\end{equation}
We have:
\begin{thm}\label{ThmMainBounds}
$K_{e^{-t\boxbt}}\in C^{\infty}\l( \l(0,\infty\r)\times M\times M \r)$.
Moreover, for every integer $j$ and ordered multi-indices $\alpha$ and
$\beta$, there is a constant $C=C\l(\alpha,\beta,j\r)$ such that:
\begin{equation*}
\l| D_x^\alpha D_y^\beta \partial_t^j K_{e^{-t\boxbt}}\l( x,y\r) \r|\leq
C \frac{\l(\rp{x}{y}\vee \sqrt{t}\r)^{-2j-\l|\alpha\r|-\l|\beta\r|}}{V\l(x,\rp{x}{y}\vee \sqrt{t}\r)}
\end{equation*}
and
\begin{equation*}
\begin{split}
&\bigg| D_x^\alpha D_y^\beta \partial_t^j K_{e^{-t\boxb}}\l( x,y\r) \bigg|\leq\\
&
\begin{cases}
C V\l( x, \rp{x}{y}\r)^{-1} 
\l( \frac{\rp{x}{y}}{t} \r)^{\l|\alpha\r|+\l|\beta\r|+2j} 
\l(\frac{\rp{x}{y}^2}{t}\r)^{Q-\frac{1}{2}}
e^{-\frac{\rp{x}{y}^2}{4\kappa^2 t}} & \text{if $t< \frac{\rp{x}{y}^2}{\kappa^2}$,}\\
C  V\l( x, \rp{x}{y} \r)^{-1}\rp{x}{y}^{-2j-\l|\alpha\r|-\l|\beta\r|} & \text{if $t\geq \frac{\rp{x}{y}^2}{\kappa^2}$}
\end{cases}  
\end{split}
\end{equation*}
\end{thm}

\begin{cor}\label{CorCleanerMainBound}
Fix $c<\frac{1}{4\kappa^2}$.  Then, for $0<t<\frac{\rp{x}{y}^2}{\kappa^2}$, we have:
\begin{equation*}
\bigg| D_x^\alpha D_y^\beta \partial_t^j K_{e^{-t\boxb}}\l( x,y\r) \bigg|\lesssim V\l( x, \rp{x}{y}\r)^{-1} \rp{x}{y}^{-\l|\alpha\r|-\l|\beta\r|-2j}e^{-c\frac{\rp{x}{y}^2}{t}}
\end{equation*}
\end{cor}

\begin{thm}[\cite{NagelSteinTheBoxbHeatEquationOnPseudoconvexManifoldsOfFiniteType}]\label{ThmHeatOpsAreNIS}
$e^{-t\boxb}$ and $e^{-t\boxbt}$ are NIS operators of order $0$ uniformly in $t>0$.  
We offer a new proof of this result.
See
Definition \ref{DefnNisOps} for the definition of NIS operators.
\end{thm}

\begin{thm}\label{ThmMultipliers}
Suppose $m:\l[0,\infty  \r)\rightarrow \C$  and $m\big|_{\l(0,\infty\r)}$ satisfies a Mihlin-H\"ormander
condition of the form:
$$\l| \l( \lambda \partial_\lambda\r)^a m\l( \lambda\r) \r|\leq C_a$$
for every $a>0$, then $m\l( \boxb\r)$ is an NIS operator of order $0$.
\end{thm}

In light of Theorem \ref{ThmMultipliers}, we see that $m\l( \boxb\r)$
is bounded on $L^p$ ($1<p<\infty$).  One expects that we do not
need an infinite amount of smoothness for $m$ to achieve this $L^p$
boundedness.  Indeed, fix $\eta\in C_0^{\infty}\l( \l(\frac{1}{2},2\r)\r)$,
with $\eta=1$ on a neighborhood of $1$, and define:
\begin{equation*}
\LtaNloc{a}{m} = \sup_{t>0} \LtaN{a}{\eta\l( \cdot\r) m\l( t\cdot\r)}
\end{equation*}
Where $\LtaN{a}{\cdot}$ denotes the usual $a$ $L^2$ Sobolev space.
(One gets essentially the same norm with any non-trivial choice of
$\eta\in C_0^\infty\l(\l( 0,\infty\r)\r)$.  See \cite{ChristLpBoundsForSpectralMultipliersOnNilpotentGroups}.)

\begin{thm}\label{ThmMultipliersLessSmooth}
Suppose $a>\frac{Q+1}{2}$, and that $\LtaNloc{a}{m}<\infty$.  Then,
$m\l( \boxb\r)$ is bounded on $L^p$, $1<p<\infty$.
\end{thm}

\begin{rmk}\label{RmkUseContVersion}
Note that, by the Sobolev embedding theorem, $m|_{\l(0,\infty\r)}$ in the statement
of Theorem \ref{ThmMultipliersLessSmooth} is equal to a continuous function.
$m\l(\boxb\r)$ is defined in terms of this continuous version of $m$.  That is,
we are not allowed to change $m$ on a set of measure $0$.
\end{rmk}

For results similar to Theorem \ref{ThmMultipliersLessSmooth}
see \cite{AlexopoulosSpectralMultipliersOnLigeGroupsOfPolynomialGrowth}
and references therein.  Indeed, the methods in that reference
are related to the methods in this paper.

\begin{rmk}
Due to the results in \cite{ChristLpBoundsForSpectralMultipliersOnNilpotentGroups}, one expects a stronger version of Theorem \ref{ThmMultipliersLessSmooth},
with $a>\frac{Q}{2}$.  This does not seem to follow directly from our
methods.  This is due to the fact that all of the proofs we know of
for multipliers that yield this sharper result take place on groups,
and use a Plancheral type theorem, for which we do not seem
to have a convenient analog.
\end{rmk}

\section{Background}
	In this section, we review the theory of NIS operators.  In addition,
we discuss the main inequality that we will use throughout the paper.

		NIS operators were first studied in \cite{NagelRosaySteinWaingerEstimatesForTheBergmanAndSzegoKernels} and the definition
we use is from \cite{KoenigOnMaixmalSobolevAndHolderEstimates}.

\begin{defn}\label{DefnNisOps}
Let $T:C_0^{\infty}\l( M\r) \rightarrow C^{\infty}\l(M\r)$ be a linear
operator, with Schwartz kernel $K_T\l( x,y\r)$.  We say $T$ is an NIS operator,
smoothing of order $r$, if $K_T$ is $C^\infty$ away from the diagonal of
$M\times M$ and the following conditions are satisfied:
\begin{enumerate}
\item For $s\geq 0$, there exist parameters $a(s)<\infty$ and $b<\infty$
such that if $\zeta,\zeta'\in C_0^\infty\l( M\r)$ with $\zeta\prec \zeta'$,
then there exists $C=C\l( s, \zeta, \zeta'\r)$ such that for all $f\in C_0^\infty\l(M\r)$,
$$\LpsN{2}{s}{\zeta T f}{M} \leq C\l( \LpsN{2}{a(s)}{\zeta' f}{M} + \LpsN{2}{b}{f}{M} \r)$$
\item There exist constants $C_{\alpha, \beta}$ such that for $x\ne y$,
$$\l|D_x^{\alpha} D_y^{\beta} K_T\l( x,y\r)\r| \leq C_{\alpha,\beta} \frac{\rp{x}{y}^{r-\l|\alpha\r|-\l|\beta\r|}}{V\l( x, \rp{x}{y}\r)}$$
\item For each integer $l\geq 0$, there is an integer $N=N\l(l\r)\geq 0$
and a constant $C=C\l(l\r)$ such that if $\phi\in C_0^{\infty}\l(B\l(x,\delta\r)\r)$, then,
\begin{equation*}
\sum_{\l|\alpha\r|=l} \l| D^{\alpha} T\l( \phi\r)\l( x\r) \r|\leq C \delta^{r-l}\sup_{y\in M}\sum_{\l| \beta\r| \leq N} \delta^{\l|\beta\r|} \l| D^\beta \phi \l
(y\r) \r|                                                                       \end{equation*}
\item The above conditions also hold for the adjoint operator $T^{*}$.          \end{enumerate}
                                                                                \end{defn}

The following results about NIS operators are well-known (see \cite{KoenigOnMaixmalSobolevAndHolderEstimates,NagelSteinTheBoxbHeatEquationOnPseudoconvexManifoldsOfFiniteType} and references therein):
\begin{itemize}                                                                 \item $\pi$ is an NIS operator of order $0$.
\item There is a self-adjoint NIS operator $\boxbt^{-1}$ of order $2$ such that $\boxb \boxbt^{-1} = 1-\pi = \boxbt^{-1}\boxb$ and $\pi\boxbt^{-1}=0=\boxbt^{-1}\pi$.
\item If $T$ is an NIS operator of order $m$, then $D^\alpha T$ and $T D^\alpha$are NIS operators of order $m-\alpha$.
\item NIS operators of order $\geq 0$ are bounded on $L^p$ ($1<p<\infty$).      For NIS operators of order $>0$ this is related to the fact that $M$ is compact.
\item NIS operators of order $0$ form an algebra.
\item If $T$ is an NIS operator of order $0$, and $\alpha$ is a fixed ordered
multi-index, then
there exist NIS operators $T_\beta$ of order $0$ such that:
$$D^\alpha T = \sum_{\l| \beta\r|\leq \l|\alpha\r|} T_{\beta} D^\beta$$
\item If $S$ is an NIS operator of order $0$ and $\alpha$ is a fixed ordered
multi-index, then there exist NIS operators $S_\beta$ of order $0$ such that:
$$S D^{\alpha} = \sum_{\l|\beta\r|=\l| \alpha\r|} D^{\beta} S_{\beta}$$
\end{itemize}

\begin{rmk}\label{RmkHomogNISCommute}
In Examples \ref{ExamplePolyModel} and \ref{ExampleQuasiHomog} below, we actually have that:
$$D^\alpha T = \sum_{\l| \beta\r|= \l|\alpha\r|} T_{\beta} D^\beta$$
For NIS operators $T$ of order $0$.
\end{rmk}

\begin{rmk}\label{RmkNisDefnsEquiv}
Property 1 of Definition \ref{DefnNisOps} was originally
(in \cite{NagelRosaySteinWaingerEstimatesForTheBergmanAndSzegoKernels})
replaced with that there existed functions $K_T^j\in C^\infty\l(M\times M\r)$
satisfying properties 2-4 uniformly such that $K_T^j\rightarrow K_T$
in $C_0^\infty\l( M\times M\r)'$.  These two definitions (at least when, say, $r=0$) turn out
to be equivalent, as was remarked to us by Ken Koenig.  Indeed, it was shown \cite{NagelSteinTheBoxbHeatEquationOnPseudoconvexManifoldsOfFiniteType}
that the identity satisfied both definitions.  Let $K^j\in C^\infty\l(M\times M\r)$
be an approximation of $\delta_{x=y}$ satisfying properties 2-4 uniformly
in $j$.  Let $T_j$ be the operator with Schwartz kernel $K_j$.  Then,
if $S$ is an operator of order $0$ in the sense of Definition
\ref{DefnNisOps}, $ST_j$ will be an appropriate smooth approximation.
The other direction is easy.
\end{rmk}

Closely related to NIS operators is the main inequality that we shall use
(it is essentially contained in Theorem 3.4.2 of \cite{NagelSteinTheBoxbHeatEquationOnPseudoconvexManifoldsOfFiniteType}):
\begin{thm}\label{ThmScaledEst}
There is a constant $R_0>0$ such that for all $R\leq R_0$ and all $f\in C^\infty
\l( M\r)$ such that $\pi f= 0$, we have that for every $\alpha$,
there exists an $L=L\l( \alpha\r)$ such that:
$$\sup_{B\l( x, R\r)} \l| D^\alpha f \r|\lesssim V\l( x, R\r)^{-\frac{1}{2}} \sum_{j=0}^L R^{2j-\l|\alpha\r|} \LpN{2}{\boxb^j f}{M}$$
\end{thm}
\begin{proof}
This theorem is closely tied to the fact that $\boxbt^{-1}$ is an NIS operator
of order $2$.  In fact, one way of showing that $\boxbt^{-1}$ is an
NIS operator of order $2$ is
to prove something like the above theorem.  Conversely, assuming $\boxbt^{-1}$
is an NIS operator of order $2$, the above theorem follows.  Indeed,
the theorem is well known for all $f\in C^{\infty}$ if $\boxb$ is replaced
by the sublaplacian $\sL:= X_1^{*}X_1+ X_2^{*}X_2$, and is proven with standard
scaling
arguments.  Now the theorem follows easily:
\begin{equation*}
\begin{split}
\sup_{B\l( x, R\r)} \l| D^\alpha f \r|\lesssim V\l( x, R\r)^{-\frac{1}{2}} \sum_{j=0}^L R^{2j-\l|\alpha\r|} \LpN{2}{\sL^j f}{M}
\end{split}
\end{equation*}
But, $\LpN{2}{\sL^j f}{M}$ is a linear combination of terms of the form
$\LpN{2}{D^\beta f}{M}$ where $\l| \beta\r|\leq 2j$.  And we see, for $f\in C^\infty$ such that $\pi f=0$:
\begin{equation*}
\begin{split}
\LpN{2}{D^{\beta} f}{M} = \LpN{2}{D^{\beta} \boxbt^{-1}\boxb f}{M}
\end{split}
\end{equation*}
But it is easy to see $D^{\beta} \boxbt^{-1} = \sum_{\l| \gamma \r| \leq \l|\beta\r| -2} A_{\gamma} D^\gamma$ where $A_\gamma$ is an NIS operator of order $\geq 0$.  Using
$L^2$ boundedness of NIS operators of order $\geq 0$, we see that:
\begin{equation*}
\LpN{2}{D^{\beta} \boxbt^{-1}\boxb f}{M} \lesssim \sum_{\l| \gamma\r|\leq \l|\beta\r|-2} \LpN{2}{D^{\beta} \boxb f}{M}
 \end{equation*}
The result now follows by induction.
\end{proof}
\begin{rmk}
In Examples \ref{ExamplePolyModel} and \ref{ExampleQuasiHomog} the same proof works, since (due to Remark \ref{RmkHomogNISCommute}) only the $L^2$ boundedness of NIS operators of order $0$ is needed.
Moreover, in these cases we may take $R_0=\infty$.
\end{rmk}

\section{On Diagonal Bounds}\label{SectionOnDiag}
	In this section we present the bounds for $K_{e^{-t\boxbt}}$, which
we call ``on diagonal bounds,'' due to the fact that 
they are analogous to the on diagonal
bounds for the classical heat operator.  This is all essentially
contained in \cite{NagelSteinTheBoxbHeatEquationOnPseudoconvexManifoldsOfFiniteType}, however we include it here as we will need some of the
side results later on, and we wish to emphasize a particular approach,
so as to make it clear how to generalize these results.
We close the section with the main part of the proof
of Theorem \ref{ThmHeatOpsAreNIS}.

First, we note that since $\partial_t e^{-t\boxb} = -\boxb e^{-t\boxb}$
and $\partial_t e^{-t\boxbt} = -\boxb e^{-t\boxbt}$ and since $\boxb$
is a polynomial of degree $2$ in $X_1,X_2$ it follows that the results
of Theorem \ref{ThmMainBounds} when $j\ne 0$ follow from the case when
$j=0$. 
For this reason, we focus only on the case $j=0$.

Second, we note that the bounds in Theorem \ref{ThmMainBounds} for $K_{e^{-t\boxb}}$ when $t\geq\frac{\rp{x}{y}^2}{\kappa^2}$
follow from those for $K_{e^{-t\boxbt}}$, the fact that $\pi$ is an NIS
operator of order $0$, and (\ref{EqnRelBetweenHeats}).  Similarly,
the bounds in Theorem \ref{ThmMainBounds} for $K_{e^{-t\boxbt}}$ when $t<\frac{\rp{x}{y}^2}{\kappa^2}$
follow from those for $K_{e^{-t\boxb}}$, the fact that $\pi$ is an NIS operator
of order $0$, and (\ref{EqnRelBetweenHeats}).
Hence, in this section, we are only concerned with the bounds for
$K_{e^{-t\boxbt}}$ when $t\geq\frac{\rp{x}{y}^2}{\kappa^2}$.

In this section, and in the rest of the paper, we will need some
elementary inequalities that are essentially contained in \cite{SikoraRieszTransformGaussianBoundsAndTheMethodOfWaveEquation} (see Equations (2.7) and
(2.11) of \cite{SikoraRieszTransformGaussianBoundsAndTheMethodOfWaveEquation}
for the below results without any derivatives).
We state these without proof.
\begin{lemma}\label{LemmaSikIneqs}
Suppose $S_1,S_2:L^2\l(M\r)\rightarrow L^2\l(M\r)$.  Fix an ordered multi-index
$\alpha$.  Suppose that for some open set $U$, $D_x^\alpha K_{S_1S_2}\l(x,y\r)\in L^1_{\loc}\l(U\times M\r)$, and that $\sup_{x\in U} \LpN{2}{D_x^{\alpha} K_{S_1}\l( x,\cdot\r)}{M}<\infty$.  Then, for $x\in U$,
\begin{equation}\label{EqnPullOutL2Bound}
\LpN{2}{D_x^{\alpha} K_{S_1 S_2}\l( x, \cdot\r)}{M} \leq \LpOpN{2}{S_2}{M} \LpN{2}{D_x^{\alpha} K_{S_1}\l( x,\cdot\r)}{M}
\end{equation}
Furthermore, if instead we have two neighborhoods, $U,V\subseteq M$, and if
$$D_x^\alpha K_{S_1}\l( x,y\r), D_y^{\beta}K_{S_2}\l( x,y\r)\in L^1_{\loc}\l(M\times M\r)$$
 with $$\sup_{x\in U}\LpN{2}{D_x^{\alpha}K_{S_1}\l( x,\cdot\r)}{M}+ \sup_{y\in V}\LpN{2}{D_y^{\beta} K_{S_2}\l(\cdot,y\r)}{M} <\infty$$
 then for $x\in U, y\in V$,
\begin{equation}\label{EqnPutIntoTwo}
\l| D_x^{\alpha} D_y^{\beta} K_{S_1S_2}\l( x,y\r) \r|\leq \LpN{2}{D_x^{\alpha}K_{S_1}\l( x,\cdot\r)}{M} \LpN{2}{D_y^{\beta} K_{S_2}\l(\cdot,y\r)}{M}
\end{equation}
\end{lemma}
\begin{rmk}
The main point of (\ref{EqnPutIntoTwo}) is that:
$$K_{S_1S_2}\l( x,y\r) = \int_{M} K_{S_1}\l( x,z\r) K_{S_2} \l( z,y\r) dz$$
\end{rmk}

Note that, by (\ref{EqnPutIntoTwo}), and using the fact that $e^{-t\boxbt}$
is self-adjoint:
\begin{equation}\label{EqnFactorBoxbt}
\begin{split}
\l| D_x^{\alpha} D_y^{\beta}K_{e^{-t\boxbt}}\l( x,y\r) \r| &\leq \LpN{2}{D_x^{\alpha} K_{e^{-\frac{t}{2}\boxbt}}\l( x,\cdot\r)}{M} \LpN{2}{D_y^{\beta} K_{e^{-\frac{t}{2}\boxbt}}\l( \cdot, y\r)}{M}\\
&= \LpN{2}{D_x^{\alpha} K_{e^{-\frac{t}{2}\boxbt}}\l( x,\cdot\r)}{M} \LpN{2}{D_y^{\beta} K_{e^{-\frac{t}{2}\boxbt}}\l( y,\cdot\r)}{M}
\end{split}
\end{equation}
Here, we have implicitly used that $K_{e^{-t\boxbt}}\in C^{\infty}\l(M\times M\r)$, which follows easily from Theorem \ref{ThmScaledEst}, since $\boxb^j e^{-t\boxbt} \boxb^k$ is bounded on $L^2\l( M\r)$ for every $j,k$, and $\pi e^{-t\boxbt} = 0 = e^{-t\boxbt}\pi$.  (\ref{EqnFactorBoxbt}) shows that to
prove the on diagonal estimates for $e^{-t\boxbt}$, the following
proposition will be sufficient:
\begin{prop}\label{PropOnDiagL2Bound}
For $t>0$, 
$$\LpN{2}{D_x^{\alpha}K_{e^{-t\boxbt}}\l( x, \cdot\r)}{M}^2\lesssim \frac{\sqrt{t}^{-2\l|\alpha\r|}}{V\l( x, \sqrt{t}\r)}$$
\end{prop}
\begin{rmk}
To see that Proposition \ref{PropOnDiagL2Bound} is sufficient,
we must use that, in this case we are concerned with, $V\l( x, \sqrt{t}\r)\approx V\l( y, \sqrt{t}\r)$,
which follows from Proposition \ref{PropHomogType} and the fact that
$\sqrt{t}\gtrsim \rp{x}{y}$.
\end{rmk}
\begin{proof}[Proof of Proposition \ref{PropOnDiagL2Bound}]
Fix $x$ and $\alpha$, and recall the number $R_0$ from Theorem \ref{ThmScaledEst}.  There are two cases:
$\sqrt{t}\leq R_0$ and $\sqrt{t}>R_0$.  We first investigate the
case $\sqrt{t}\leq R_0$.  Let $\phi\in C^\infty\l(M\r)$.  We apply
Theorem \ref{ThmScaledEst} to see that there exists an $L$ such that:
\begin{equation*}
\begin{split}
\l| \l(D_x^{\alpha} e^{-t\boxbt} \phi\r)\l( x\r)\r| & \lesssim V\l( x, \sqrt{t}\r)^{-\frac{1}{2}} \sum_{j=0}^L \sqrt{t}^{2j-\l| \alpha\r|} \LpN{2}{\boxb^j e^{-t\boxbt}\phi}{M}\\
&= V\l( x, \sqrt{t}\r)^{-\frac{1}{2}} \sum_{j=0}^L \sqrt{t}^{-\l| \alpha\r|} \LpN{2}{\l(t\boxb\r)^j e^{-t\boxbt}\phi}{M}\\
&\lesssim V\l( x, \sqrt{t}\r)^{-\frac{1}{2}} \sqrt{t}^{-\l| \alpha\r|} \LpN{2}{\phi}{M}
\end{split}
\end{equation*}
where in the last line, we have used that $\l( t\boxb\r) e^{-t\boxbt}$ is bounded on $L^2$ uniformly in $t$.  Taking the supremum over all $\LpN{2}{\phi}{M}=1$,
we see that the statement of the proposition follows in this case.

If $R_0=\infty$, we would be done.  Since in this case $R_0$ may not
equal $\infty$, we use the fact the $\boxbt^{-1}:L^2\l(M\r)\rightarrow L^2\l(M\r)$ (a fact we do not have in all of the examples we consider in Section \ref{SectionOtherExamples},
however, in those examples where we do not have it, we instead have $R_0=\infty$).

We now assume $\sqrt{t}>R_0$.  We use that $V\l(x,R_0\r)\approx 1\approx V\l( x, \sqrt{t}\r)$ for all $x$.  We apply the above proof with $R_0$
in place of $\sqrt{t}$ to see:
\begin{equation*}
\begin{split}
\l| \l(D_x^{\alpha} e^{-t\boxbt} \phi\r)\l( x\r)\r| & \lesssim V\l( x, R_0\r)^{-\frac{1}{2}} \sum_{j=0}^L R_0^{2j-\l| \alpha\r|} \LpN{2}{\boxb^j e^{-t\boxbt}\phi}{M}\\
& \approx \sum_{j=0}^L \LpN{2}{\boxb^j e^{-t\boxbt}\phi}{M}\\
&=\sum_{j=0}^L \LpN{2}{\boxbt^{-1} \boxb^{j+1} e^{-t\boxbt}\phi}{M}\\
&\lesssim \sum_{j=0}^L \LpN{2}{\boxb^{j+1} e^{-t\boxbt}\phi}{M}\\
&\lesssim \cdots \lesssim \sum_{j=0}^L \LpN{2}{\boxb^{j+N} e^{-t\boxbt}\phi}{M}\\
&\lesssim t^{-N} \LpN{2}{\phi}{M}\\
&\lesssim V\l( x, \sqrt{t}\r)^{-\frac{1}{2}} \sqrt{t}^{-\l|\alpha\r|} \LpN{2}{\phi}{M}
\end{split}
\end{equation*}
provided $2N\geq \l|\alpha\r|$, which completes the proof.
\end{proof}
\begin{rmk}
One might note that following the same method we used when $\sqrt{t}>R_0$,
one could use the $L^2\l( M\r)$ boundedness of $\boxbt^{-1}$ to show that:
$$\LpN{2}{D_x^{\alpha}K_{e^{-t\boxbt}}\l( x, \cdot\r)}{M}^2\lesssim \frac{\sqrt{
t}^{-2N}}{V\l( x, \sqrt{t}\r)}$$
for all $N\geq \l|\alpha\r|$.  Indeed, this is not surprising given
that the spectrum of $\boxb$ is discrete, and we actually expect
the above bound to fall off exponentially in $t$.  This will follow
from our study of multipliers in Section \ref{SectionMultipliers}.  However, we do not
expect this sort of decay in (say) Examples \ref{ExamplePolyModel} and \ref{ExampleQuasiHomog} below, where
$0$ is not an isolated point of the spectrum of generator of the
heat semigroup.
\end{rmk}

From Proposition \ref{PropOnDiagL2Bound} we obtain
the following:
\begin{prop}\label{PropOnDiagPolyBound}
For every $\alpha$ an ordered multi-index, and every $m>Q+2\l|\alpha\r|$
we have:
$$\LpN{2}{D_x^\alpha K_{\l(1+t\boxbt\r)^{-\frac{m}{4}}}\l( x,\cdot\r)}{M}\lesssim \frac{\sqrt{t}^{-\l|\alpha\r|}}{V\l( x, \sqrt{t}\r)^{\frac{1}{2}}}$$
\end{prop}
\begin{proof}
This follows just as in Theorem 1 of \cite{SikoraRieszTransformGaussianBoundsAndTheMethodOfWaveEquation}, by using Proposition \ref{PropOnDiagL2Bound}.
\end{proof}

We now close this section with the main part of the proof of Theorem \ref{ThmHeatOpsAreNIS}.
Indeed we will show that property 3 of Definition \ref{DefnNisOps} holds
uniformly in $t> 0$ for $e^{-t\boxbt}$.  (\ref{EqnRelBetweenHeats}),
along with the fact that $\pi$ is an NIS operator of order $0$,
then shows that the same is true for $e^{-t\boxb}$.
Fix an ordered multi-index $\alpha$, and apply Theorem \ref{ThmScaledEst}
to see, for $\phi\in C_0^{\infty}\l( B\l( x, \delta\r)\r)$ (with $\delta\leq R_0$):
\begin{equation*}
\begin{split}
\l| D^{\alpha} e^{-t\boxbt} \phi \l( x\r)\r| &\lesssim V\l(x, \delta\r)^{-\frac{1}{2}} \sum_{j=0}^L \delta^{2j-\l|\alpha\r|} \LpN{2}{\boxb^j e^{-t\boxbt} \phi}{M}\\
&\lesssim V\l(x, \delta\r)^{-\frac{1}{2}} \sum_{j=0}^L \delta^{2j-\l|\alpha\r|}\LpN{2}{\boxb^j \phi}{M}\\
&\lesssim \sum_{\l|\beta\r|\leq 2L} \delta^{\l|\beta\r|-\l|\alpha\r|} \sup_{B\l(x,\delta\r)} \l| D^\beta \phi \r|
\end{split}
\end{equation*}
The case when $\delta>R_0$ follows just as above, but by using $R_0$ in place
of $\delta$.

Once we have established the off diagonal bounds for $e^{-t\boxb}$,
the remainder of Theorem \ref{ThmHeatOpsAreNIS} will follow immediately.  We
leave the details to the reader.

\section{Finite Speed of Propagation}\label{SectionFiniteSpeed}
	In this section, we discuss Theorem \ref{ThmFiniteProp}, which will
be one of our main tools in proving the off diagonal bounds for
$e^{-t\boxb}$.  The results in \cite{MelrosePropagationForTheWaveGroupOfAPositiveSubellipticSecondOrderDifferentialOperator} are stated
in terms of a metric which (in this case) is defined in terms of
$\boxb$.  It is easy to see that $\boxb$ and
the sublaplacian $X_1^{*}X_1+X_2^{*}X_2$ give rise to the
same metric.  It is well known that the metric in \cite{MelrosePropagationForTheWaveGroupOfAPositiveSubellipticSecondOrderDifferentialOperator}
induced by the sublaplacian is equivalent to the metric $\rho$.
Then Theorem \ref{ThmFiniteProp} follows from \cite{MelrosePropagationForTheWaveGroupOfAPositiveSubellipticSecondOrderDifferentialOperator}.

The above outline gives rise to a somewhat round-about proof.
Indeed, the result in \cite{MelrosePropagationForTheWaveGroupOfAPositiveSubellipticSecondOrderDifferentialOperator}
follows by approximating $\boxb$ by elliptic operators, and proving
a corresponding result for the approximating elliptic operators.
Then one must note the above equivalence of metrics.  In
the special case of $\boxb$, however, a more direct proof
will suffice.  Indeed, we need only adapt the 
proof on page 162 of \cite{FollandIntroductionToPartialDifferentialEquations} and the one in the appendix of \cite{MullerMarcinkiewiczMultipliersAndMultiParameterStructureOnHeisenbergGroupsLectureNotes} to this situation, and we present this argument below.  
One benefit of this argument is that it requires essentially no
work to adapt it to Examples \ref{ExamplePseudoConvex} and \ref{ExamplePolyModel}, below;
though these cases can also be covered by the methods of
\cite{MelrosePropagationForTheWaveGroupOfAPositiveSubellipticSecondOrderDifferentialOperator}, with a little more work.
To proceed, we need a result from \cite{NagelSteinDifferentiableControlMetricsAndScaledBumpFunctions}:
\begin{prop}[\cite{NagelSteinDifferentiableControlMetricsAndScaledBumpFunctions}]\label{PropSmoothMetric}
There exists a function $d:M\times M\rightarrow \R^{+}$ such that:
$$d\l(x,y\r) \approx \rp{x}{y}$$
and for $x\ne y$,
$$\l| D_x^{\alpha}D_y^{\beta} d\l( x,y\r) \r|\lesssim d\l( x, y\r)^{1-\l|\alpha\r|-\l|\beta\r|}$$
By replacing $d\l( x,y\r)$ with $d\l( x,y\r)+d\l(y,x\r)$ we 
may assume that $d\l(x,y\r)=d\l(y,x\r)$.
By multiplying $d$ by a fixed constant, we may also assume:
$$\sup_{x\ne y}\sum_{\l|\alpha\r|=1} \l| D_y^{\alpha} d\l( x,y\r)\r| \leq 1 $$
\end{prop}

\begin{prop}\label{PropFiniteSpeed}
Suppose $u\l( x,t\r)\in C^2\l( M\times\l[ 0,T\r] \r)$
such that $\partial_t^2 u + \boxb u =0$, and $u=\partial_t u=0$ on the ball
$$\l\{ \l(y,0\r) :  d\l( x_0,y \r)\leq t_0 \r\}$$
where $x_0\in M$ and $0<t_0\leq T$.  Then $u$ vanishes in the region:
$$\Omega = \l\{ \l(y,t\r) : 0\leq t \leq t_0, d\l( x_0,y\r)\leq \l( t_0-t\r)  \r\}$$
\end{prop}
\begin{proof}
Given $\delta>0$, small, let $\chi_\delta\in C_0^{\infty}\l( \R\r)$ be such that
$\chi_\delta\l( s\r)$ is equal to $1$ when $s\leq 1$, $\chi_\delta\l( s\r) = 0$ for $s\geq 1+\delta$, and $\chi_\delta'\leq 0$, and let $\chi_0$ be the characteristic function
of $\l(-\infty,1\r]$.  
Note that $\lim_{\delta\rightarrow 0} \chi_\delta = \chi_0$, pointwise.
Define, for $\delta\geq 0$:
$$E_{\delta}\l( t\r) = \frac{1}{2} \int \l(\l| \partial_t u\l( y,t\r)\r|^2 + \l| \l( X_1+iX_2\r)u(y,t)\r|^2\r) \chi_\delta\l( \frac{d\l(x_0,y\r)}{t_0-t}\r) dy$$
where $X_1$ and $X_2$ are acting in the $y$ variable.

Consider, writing $\ip{a}{b}=a\overline{b}$, we have for $\delta>0$:
\begin{equation*}
\begin{split}
\frac{dE_\delta}{dt}\l( t\r) &= \Re \int \l( \ip{u_{tt}}{u_t} + \ip{\l(X_1+iX_2\r) u}{\l( X_1+iX_2\r) u_t}  \r) \chi_\delta\l( \frac{d\l( x_0,y\r)}{t_0-t} \r) dy
\\ &\qquad+  \frac{1}{2}\int \l(\l| u_t\r|^2 + \l|\l(X_1+iX_2\r) u \r|^2\r) \chi_\delta'\l( \frac{d\l( x_0,y\r)}{t_0-t} \r)\frac{d\l( x_0,y\r)}{\l(t_0-t\r)^2} dy\\
&=: I+ II
\end{split}
\end{equation*}
Since $\chi_\delta'\leq 0$, $II$ is clearly negative.  We will show
that (for $t< t_0$), $\l| I\r| \leq \l| II\r|$, and then it will follow that
$\frac{dE_\delta}{dt}\l(t\r)\leq 0$.

We use the fact that $X_1^{*}=-X_1+g$, where $g\in C^\infty\l( M\r)$
and similarly for $X_2$ to see that:
\begin{equation*}
\begin{split}
\l| I\r| &\leq \l| \int \l( \ip{u_{tt}}{u_t} + \ip{\l( -X_1^{*} +i X_2^{*}\r) \l( X_1+iX_2\r) u}{u_t} \r) \chi_\delta\l( \frac{d\l(x_0,y\r)}{t_0-t} \r)  dy \r|\\
&\quad + \sum_{\l|\alpha\r|=1}\l| \int \l( \ip{\l( X_1+iX_2\r) u}{u_t} \r)\l(  D_y^{\alpha} \l(\chi_\delta\l(\frac{d\l( x_0,y\r)}{t_0-t}\r)\r)\r) dy  \r|\\
&=: III+IV
\end{split}
\end{equation*}
Now the integrand of $III$ contains the term:
\begin{equation*}
\ip{u_{tt}}{u_t} + \ip{\l( -X_1^{*} +i X_2^{*}\r) \l( X_1+iX_2\r) u}{u_t} = \ip{u_{tt} + \boxb u}{u_t} =0
\end{equation*}
and it follows that $III=0$.  To bound $IV$, note that:
\begin{equation*}
\begin{split}
\l| \sum_{\l|\alpha\r|=1} D_y^{\alpha} \l(\chi_\delta\l( \frac{d\l( x_0,y\r)}{t_0-t} \r)\r)\r|&\leq  - \chi_\delta' \l( \frac{d\l( x_0,y\r)}{t_0-t} \r) \sum_{\l|\alpha\r|=1}\l|\frac{D_y^{\alpha} d\l( x_0,y\r)}{t_0-t} \r|\\
&\leq - \chi_\delta' \l( \frac{d\l( x_0,y\r)}{t_0-t} \r) \frac{d\l( x_0,y\r)}{\l(t_0-t\r)^2}
\end{split}
\end{equation*}
In the last line, we used that $\frac{d\l(x_0,y\r)}{t_0-t}\geq 1$
on the support of $\chi_\delta' \l( \frac{d\l( x_0,y\r)}{t_0-t} \r)$.

Thus, we have that:
\begin{equation*}
\begin{split}
IV &\leq \int \l| \ip{\l(X_1+iX_2\r)u}{u_t} \r|\l(-\chi_\delta' \l( \frac{d\l( x_0,y\r)}{t_0-t} \r) \r) \frac{d\l( x_0,y\r)}{\l(t_0-t\r)^2} \\
&\leq \frac{1}{2}\int \l(\l| u_t\r|^2 + \l| \l(X_1+iX_2\r)u \r|^2 \r) \l(-\chi_\delta' \l( \frac{d\l( x_0,y\r)}{t_0-t} \r) \r) \frac{d\l( x_0,y\r)}{\l(t_0-t\r)^2}\\
&= \l|II\r|
\end{split}
\end{equation*}
Hence, $\l| I\r| \leq \l|II\r|$, and therefore $\frac{dE_\delta}{dt}\l( t\r)\leq 0$
for all $0\leq t<t_0$.  It follows that 
\begin{equation}\label{EqnEdeltaDecreasing}
E_\delta\l( t\r) \leq E_\delta\l(0\r)
\end{equation}
for all $0\leq t\leq t_0$.  Taking the limit of both sides of
(\ref{EqnEdeltaDecreasing}) as $\delta\rightarrow 0$ and applying
dominated convergence, we see:
$$E_0\l( t\r) \leq E_0\l(0\r)$$
for all $0\leq t\leq t_0$.

However, by our assumption on $u$, $E_0\l( 0\r)=0$.  It follows that
$E_0\l(t\r)=0$ for all $0\leq t\leq t_0$, and in particular $\partial_t u=0$
on $\Omega$.  It follows that $u\l(y,t\r)=0$ on $\Omega$.
\end{proof}
Now Theorem \ref{ThmFiniteProp} will follow immediately from the
following corollary:
\begin{cor}
For $t>0$,
$$\supp\l(K_{\cos\l(t\sqrt{\boxb}\r)}\r)\subseteq \l\{ \l(x,y\r): d\l( x,y\r) \leq t \r\}$$
\end{cor}
\begin{proof}
For this proof, define $B_d\l( x, \delta\r)=\l\{y\in M: d\l( x,y\r) <\delta\r\}$.
Fix $x_0,y_0\in M$ and $t_1>0$ such that $d\l( x_0,y_0\r) >t_1$.  Fix $\epsilon>0$ so small that
for all $x\in B_d\l( x_0, \epsilon\r)$ and $y\in B_d\l( y_0, \epsilon\r)$, we
have $d\l( x,y\r) >t_1+\epsilon$.  We will show that for every
$\phi\in C_0^\infty\l(B_d\l(x_0, \epsilon\r)\r)$, $\psi\in C_0^\infty\l(B_d\l(y_0,\epsilon\r) \r)$,
we have:
$$\int \psi\l(z\r) \l(\cos\l(t_1\sqrt{\boxb}\r) \phi\r)\l(z\r) dz = 0 $$
and the claim will follow.

Define $u\l( x, t\r) = \cos\l( t\sqrt{\boxb}\r)\phi$.  We first claim
that $u\in C^\infty\l( M\times \R\r)$.  Note that,
\begin{equation*}
u  = \pi \phi + \cos\l( t\sqrt{\boxbt}\r) \phi =: u_1+u_2
\end{equation*}
$u_1$ is independent of $t$ and is $C^\infty$ since $\pi$ is an NIS operator.
Fix $t$ and note that $\pi u_2\l(\cdot, t\r)=0$.  Thus, since
$$\boxb^j u_2 = \cos\l( t\sqrt{\boxbt}\r) \boxb^j \phi\in L^2\l( M\r)$$
for each fixed $t$, we have that $u_2\l( t, \cdot\r)$ is in $C^\infty$
for each fixed $t$.  Moreover, since $\partial_t^{4N} u_2 = \boxb^{2N} u_2\in L^2\l(M\r)$
for every $N$ (in distribution), the usual Fourier inversion trick now
shows that $u_2\in C^\infty\l(M\times \R \r)$.  It follows that $u\in C^\infty\l( M\times \R\r)$.

Thus, we are in a position to apply Proposition \ref{PropFiniteSpeed}.
Note that $u\l(y,0\r)= 0 =\partial_t u\l( y,0\r)$ for all $y\in B_d\l( y_0, t_1+\epsilon\r)$.  Taking $t_0=t_1+\epsilon$ in Proposition \ref{PropFiniteSpeed}
and taking the $t=t_1$ slice of $\Omega$ we see that:
$$u\l( y ,t_1\r)=0$$
for all $y\in B_d\l( y_0, \epsilon\r)$.  Hence,
$$\int \psi\l(y\r) u\l( t_1,y\r) dy=0$$
completing the proof.
\end{proof}

\begin{rmk}
Note that $\cos\l( t\sqrt{\boxbt} \r)$ does not have finite 
propagation speed.  This is essential in understanding why we do not
get off-diagonal Gaussian bounds for $e^{-t\boxbt}$.
\end{rmk}

\begin{cor}\label{CorFourierTransform0}
Suppose $\widehat{F}$ is the Fourier transform of an even bounded Borel
function $F$ with $\supp \widehat{F}\subseteq \l[-r,r\r]$.  Then,
$$\supp \l(K_{F\l( \sqrt{\boxb} \r)}\r) \subseteq \l\{ \l(x,y\r) : \rp{x}{y}\leq \kappa r \r\}$$
\end{cor}
\begin{proof}
This follows just as Lemma 3 of \cite{SikoraRieszTransformGaussianBoundsAndTheMethodOfWaveEquation}, using Theorem \ref{ThmFiniteProp}.
\end{proof}

\section{Off Diagonal Bounds}\label{SectionOffDiag}
	In this section, we complete the proof of Theorem \ref{ThmMainBounds}
by proving the bounds on $K_{e^{-t\boxb}}\l( x,y\r)$ when $t<\frac{\rp{x}{y}^2}{\kappa^2}$; to do so, we modify the proof of Theorem 4 of \cite{SikoraRieszTransformGaussianBoundsAndTheMethodOfWaveEquation}.  We use the same notation
as \cite{SikoraRieszTransformGaussianBoundsAndTheMethodOfWaveEquation} to
make our modifications obvious.
Fix $x_0,y_0\in M$, $t>0$, with $t<\frac{\rp{x_0}{y_0}^2}{\kappa^2}$.

Fix a function $\psi\in C^\infty\l( \R\r)$, satisfying
\begin{equation*}
\psi\l(x\r) =
\begin{cases}
0 & \text{ if $x\leq -1$,}\\
1 & \text{ if $x\geq -\frac{1}{2}$}
\end{cases}
\end{equation*}
and for $s>1$ define $\phi_s\l( x\r) = \psi\l( s\l( \l|x\r|-s\r)\r)$.
Define:
\begin{equation*}
\begin{split}
F_s\l( x \r) &= \phi_s\l( x\r) \frac{1}{\sqrt{4\pi}} e^{-\frac{x^2}{4}}\\
R_s\l( x\r) &= \l( 1- \phi_s\l( x\r)\r) \frac{1}{\sqrt{4\pi}} e^{-\frac{x^2}{4}}
\end{split}
\end{equation*}
so that $\widehat{F}_s\l( \lambda\r) + \widehat{R}_s\l( \lambda\r) = e^{-\lambda^2}$ (here $\widehat{F}_s$ denotes the Fourier transform of $F_s$)
and therefore $\widehat{F}_s\l( \sqrt{t\boxb}\r) + \widehat{R}_s\l( \sqrt{t\boxb}\r) = e^{-t\boxb}$.
In \cite{SikoraRieszTransformGaussianBoundsAndTheMethodOfWaveEquation}
equation (5.2) it is shown that for every natural number $N$ there
exists a $C$ such that for all $s>1$,
\begin{equation}\label{EqnBoundForFs}
\l|\widehat{F}_s\l(\lambda\r) \r| \leq C \frac{1}{s\l(1+\frac{\lambda^2}{s^2}\r)^N}e^{-\frac{s^2}{4}}
\end{equation}
Note, also, that $\supp\l( R_s\r) \subseteq \l[ -s+\frac{1}{2s}, s-\frac{1}{2s}\r]$ and thus if we set $s_{x_0,y_0} = \frac{\rp{x_0}{y_0}}{\kappa\sqrt{t}}$,
Corollary \ref{CorFourierTransform0} tells us that
$$D_x^{\alpha} D_y^{\beta} K_{\widehat{R}_{s_{x_0,y_0}}\l(\sqrt{t\boxb}\r)}\l( x, y\r)\bigg|_{\substack{x=x_0\\y=y_0}} =0$$
Thus,
\begin{equation}\label{EqnTwoParts}
\begin{split}
&D_x^\alpha D_y^\beta K_{e^{-t\boxb}}\l(x,y\r)\bigg|_{\substack{x=x_0\\y=y_0}}= D_x^{\alpha}D_y^{\beta}K_{\widehat{F}_{s_{x_0,y_0}}\l(\sqrt{t\boxb}\r)}\l(x,y\r)\bigg|_{\substack{x=x_0\\y=y_0}}\\
&\quad= D_x^{\alpha} D_y^{\beta} K_{\widehat{F}_{s_{x_0,y_0}}\l(\sqrt{t\boxbt}\r)}\l(x,y\r)\bigg|_{\substack{x=x_0\\y=y_0}} + \widehat{F}_{s_{x_0,y_0}}\l( 0\r) D_x^{\alpha} D_y^{\beta} K_\pi\l( x,y\r)\bigg|_{\substack{x=x_0\\y=y_0}}
\end{split}
\end{equation}
We bound these two terms separately.\footnote{The second term is the main difference
between this proof, and the one in \cite{SikoraRieszTransformGaussianBoundsAndTheMethodOfWaveEquation}.}
First, we start with the second term.  We use the bound (\ref{EqnBoundForFs})
with $\lambda=0$ and the fact that $\pi$ is an NIS operator of order $0$ to see:
\begin{equation}\label{EqnBoundSzegoPart}
\begin{split}
&\l|\widehat{F}_{s_{x_0,y_0}}\l( 0\r) D_x^{\alpha} D_y^{\beta} K_\pi\l( x,y\r)\bigg|_{\substack{x=x_0\\y=y_0}}\r| \lesssim \frac{1}{s_{x_0,y_0}} e^{-\frac{s_{x_0,y_0}^2}{4}} \frac{\rp{x_0}{y_0}^{-\l| \alpha\r|-\l|\beta\r|}}{V\l( x_0, \rp{x_0}{y_0}\r)}\\
&\qquad\approx \frac{\sqrt{t}}{\rp{x_0}{y_0}} e^{-\frac{\rp{x_0}{y_0}^2}{4\kappa^2t}} \frac{\rp{x_0}{y_0}^{-\l| \alpha\r|-\l|\beta\r|}}{V\l( x_0, \rp{x_0}{y_0}\r)}\\
&\qquad\lesssim \frac{\sqrt{t}}{\rp{x_0}{y_0}} \l(\frac{\rp{x_0}{y_0}}{t}\r)^{\l|\alpha\r|+\l|\beta\r|}\frac{1}{V\l( x_0, \rp{x_0}{y_0}\r)}e^{-\frac{\rp{x_0}{y_0}^2}{4\kappa^2t}}
\end{split}
\end{equation}
where in the last line, we have used the fact that
$$\frac{1}{\rp{x_0}{y_0}}\lesssim \frac{\rp{x_0}{y_0}}{t}$$
(\ref{EqnBoundSzegoPart}) is even better than the bound in the conclusion of
Theorem \ref{ThmMainBounds}.

We now turn to the first term in the last line of (\ref{EqnTwoParts}).
Let $J\l( \lambda\r)$ be a measurable function such that
$J\l( \lambda\r)^2 =\widehat{F}_{s_{x_0,y_0}}\l( t^{\frac{1}{2}}\lambda\r)$ (we suppress
$J$s dependence on $s_{x_0,y_0}$).  Note that, by (\ref{EqnBoundForFs}),
we have for every $N\geq 0$:
\begin{equation}\label{EqnBoundForJs}
\sup_{\lambda\geq 0} \l| J\l(\lambda\r) \l(1+ \frac{\lambda^2 t^2}{\rp{x_0}{y_0}^2}\r)^N \r|\lesssim \frac{1}{\sqrt{\rp{x_0}{y_0}t^{-\frac{1}{2}}}} e^{-\frac{\rp{x_0}{y_0}^2}{8\kappa^2 t}}
\end{equation}
Then we have, by Lemma \ref{LemmaSikIneqs} and Proposition \ref{PropOnDiagPolyBound} (taking $N$ large, depending on $\gamma$):
\begin{equation*}
\begin{split}
&\LpN{2}{D_x^{\gamma} K_{J\l(\sqrt{\boxbt}\r)}\l( x, \cdot\r)}{M} \lesssim \frac{e^{-\frac{\rp{x_0}{y_0}^2}{8\kappa t}}}{\sqrt{\rp{x_0}{y_0}t^{-\frac{1}{2}}}} \LpN{2}{D_x^\gamma K_{\l(I+\frac{t^2\boxbt}{\rp{x_0}{y_0}^2}\r)^{-N}}\l(x,\cdot\r)}{M}\\
&\qquad \lesssim \frac{e^{-\frac{\rp{x_0}{y_0}^2}{8\kappa t}}}{\sqrt{\rp{x_0}{y_0}t^{-\frac{1}{2}}}} \l(\frac{\rp{x_0}{y_0}}{t}\r)^{\l|\gamma\r|} V\l( x, \frac{t}{\rp{x_0}{y_0}}\r)^{-\frac{1}{2}}
\end{split}
\end{equation*}
Applying Lemma \ref{LemmaSikIneqs} again, we see:
\begin{equation*}
\begin{split}
&\l| D_x^{\alpha}D_y^\beta K_{\widehat{F}_{s_{x_0,y_0}}\l(\sqrt{t\boxbt}\r)}\l(x,y\r)\bigg|_{\substack{x=x_0\\y=y_0}} \r| \\
&\quad \lesssim \LpN{2}{D_x^{\alpha} K_{J\l(\sqrt{\boxbt}\r)}\l(x,\cdot\r)\bigg|_{x=x_0}}{M} \LpN{2}{D_y^{\beta} K_{J\l(\sqrt{\boxbt}\r)}\l(y,\cdot\r)\bigg|_{y=y_0}}{M}\\
&\quad \lesssim \frac{e^{-\frac{\rp{x_0}{y_0}^2}{4\kappa t}}}{\rp{x_0}{y_0}t^{-\frac{1}{2}}}
\l(\frac{\rp{x_0}{y_0}}{t}\r)^{\l| \alpha\r| + \l|\beta\r|} V\l(x_0, \frac{t}{\rp{x_0}{y_0}}\r)^{-\frac{1}{2}} V\l( y_0,\frac{t}{\rp{x_0}{y_0}}\r)^{-\frac{1}{2}}
\end{split}
\end{equation*}
Now Theorem \ref{ThmMainBounds} will follow directly from
the following lemma:
\begin{lemma}
\begin{equation*}
V\l(x_0, \frac{t}{\rp{x_0}{y_0}}\r)^{-1}, V\l( y_0,\frac{t}{\rp{x_0}{y_0}}\r)^{-1}\lesssim \l(\frac{\rp{x_0}{y_0}^2}{t}\r)^Q\frac{1}{V\l(x_0, \rp{x_0}{y_0}\r)}
\end{equation*}
\end{lemma}
\begin{proof}
For $V\l(x_0, \frac{t}{\rp{x_0}{y_0}}\r)^{-1}$ this follows directly from
Proposition \ref{PropHomogType}.  For $V\l( y_0,\frac{t}{\rp{x_0}{y_0}}\r)^{-1}$
this follows from Proposition \ref{PropHomogType} and the fact that
$$V\l( y_0,\rp{x_0}{y_0}\r)\approx V\l(x_0,\rp{x_0}{y_0}\r)$$
(Which is also a consequence of Proposition \ref{PropHomogType}.)
\end{proof}
Corollary \ref{CorCleanerMainBound} is a simple corollary of 
Theorem \ref{ThmMainBounds}.


\section{Multipliers}\label{SectionMultipliers}
	In this section, we prove Theorems \ref{ThmMultipliers} and \ref{ThmMultipliersLessSmooth}.  We prove the two in tandem, as some of
the estimates needed for Theorem \ref{ThmMultipliersLessSmooth}
are sharper than those needed for Theorem \ref{ThmMultipliers},
however Theorem \ref{ThmMultipliers} will allow us to create
a convenient Littlewood-Paley square function, with which we will
complete the proof of Theorem \ref{ThmMultipliersLessSmooth}.
The arguments in this section are closely related to those
of \cite{MullerMarcinkiewiczMultipliersAndMultiParameterStructureOnHeisenbergGroupsLectureNotes}, however, since we are not in the case of a stratified
group, we are forced to deviate from those arguments (in particular,
this is why we have $\frac{Q+1}{2}$ in Theorem \ref{ThmMultipliersLessSmooth},
instead of $\frac{Q}{2}$).

\begin{prop}\label{PropMultCptSupp}
Suppose $m$ is supported in $\l[\frac{1}{4},4\r]$, $r>0$ is fixed, $\alpha,\beta$ are fixed ordered
multi-indices, $a>\frac{Q+1}{2}+\l|\alpha\r|\vee\l|\beta\r|$,
and $\LtaN{a}{m}<\infty$.  Then, for every $\frac{Q}{2}+\l|\alpha\r|\vee \l|\beta\r|< b\leq a-\frac{1}{2}$, there exists a $C=C\l(\alpha,\beta, a, b, \LtaN{a}{m}\r)$, but not depending on $r$, such that:
\begin{equation*}
\l| D_x^\alpha D_y^\beta K_{m\l(r^2\boxb\r)} \l(x,y\r)\r|\leq C \l( 1+\frac{\rp{x}{y}}{r}\r)^{-a+\frac{1}{2}+b}\frac{\l( r\vee \rp{x}{y}\r)^{-\l|\alpha\r|-\l|\beta\r|}}{V\l(x, \rp{x}{y}+r\r)}
\end{equation*}
\end{prop}
\begin{proof}
Fix $x_0,y_0\in M$.  We wish to bound:
\begin{equation*}
D_x^\alpha D_y^\beta K_{m\l(r^2\boxb\r)}\l(x,y\r)\bigg|_{\substack{x=x_0\\ y=y_0}}
\end{equation*}
We begin with the harder case $\rp{x_0}{y_0}\geq r$.  Define $\psi\l( \lambda\r) = m\l( \lambda^2\r)$, so that $\LtaN{a}{\psi}\lesssim \LtaN{a}{m}$ (due
to the support of $m$).  Let $\psi_r\l(\lambda\r) = \psi\l( r\lambda\r)$,
so that $\psi_r\l( \boxb\r) = m\l( r^2\boxb\r)$, and $\widehat{\psi_r}\l( \xi\r) = \frac{1}{r}\widehat{\psi}\l( \frac{\xi}{r}\r)$.  Let $\phi\in C^\infty\l( \R\r)$ be such that:
\begin{equation*}
\phi\l( \xi \r) =
\begin{cases}
0 & \text{ if $\l|\xi\r|\leq \frac{1}{4}$}\\
1 & \text{ if $\l|\xi\r|\geq \frac{1}{2}$}
\end{cases}
\end{equation*}
For $s>0$, define:
\begin{equation*}
\widehat{F}_s\l( \xi\r) = \phi\l( \frac{\xi}{s}\r) \frac{1}{r}\widehat{\psi}\l( \frac{\xi}{r}\r)
\end{equation*}
where we have suppressed the dependence of $F$ on $r$.  We will use the following elementary fact:  for all $0\leq \bt\leq a-\frac{1}{2}$,
\begin{equation}\label{EqnMultiplierBoundF}
\sup_{s>0,r>0,\lambda>0} \l(1+\lambda s \r)^{\bt} \l(1+\frac{s}{r}\r)^{a-\frac{1}{2}-\bt}\l| F_s\l(\lambda\r) \r|\lesssim 1
\end{equation}
Here, the implicit constant depends on the same parameters
as in the statement of the proposition.  We leave the proof of (\ref{EqnMultiplierBoundF}) to the interested reader.

Set $s = \frac{\rp{x_0}{y_0}}{\kappa}$, so that $\frac{s}{r}\gtrsim 1$.
Note that, by the definition of $\widehat{F}_s$, Theorem \ref{ThmFiniteProp},
and Corollary \ref{CorFourierTransform0}, we have:
\begin{equation}\label{EqnMutliplerTwoParts}
\begin{split}
&D_x^\alpha D_y^\beta K_{m\l(r^2\boxb\r)}\l( x,y\r)\bigg|_{\substack{x=x_0\\y=y_0}} = D_x^\alpha D_y^\beta K_{F_s\l(\sqrt{\boxb}\r)}\l(x,y\r)\bigg|_{\substack{x=x_0\\y=y_0}}
\\ & \qquad =D_x^{\alpha}D_y^\beta K_{F_s\l(\sqrt{\boxbt}\r)}\l(x,y\r)\bigg|_{\substack{x=x_0\\y=y_0}} + F_s\l( 0\r) D_x^{\alpha}D_y^\beta K_\pi\l(x,y\r)\bigg|_{\substack{x=x_0\\y=y_0}}
\end{split}
\end{equation}
We bound the two terms in the last line of (\ref{EqnMutliplerTwoParts}) separately.
For the second term, we apply (\ref{EqnMultiplierBoundF}) (with $\bt=0$) and the fact
that $\pi$ is an NIS operator of order $0$ to see:
\begin{equation*}
\l|F_s\l( 0\r) D_x^\alpha D_y^\beta K_\pi\l(x,y\r)\bigg|_{\substack{x=x_0\\y=y_0}} \r| \lesssim \l( 1+\frac{\rp{x_0}{y_0}}{r}\r)^{-a+\frac{1}{2}} \frac{\rp{x_0}{y_0}^{-\l|\alpha\r|-\l|\beta\r|}}{V\l(x,\rp{x_0}{y_0}\r)}
\end{equation*}
which, in light of the fact that $\rp{x_0}{y_0}\geq r$, is at least as good as
the bound in the statement of the proposition.
We now turn to the first term in the last line of (\ref{EqnMutliplerTwoParts}).
We let $J_s\l(\lambda\r)$ be a measurable function such that $J_s\l(\lambda\r)^2=F_s\l(\lambda\r)$ (we again suppress the dependence on $r$, and note that
all of our bounds will be uniform in $r$); we now have, from (\ref{EqnMultiplierBoundF}), with $\bt=b$:
\begin{equation*}
\sup_{s>0,r>0,\lambda>0} \l(1+\lambda s\r)^{\frac{b}{2}} \l(1+\frac{s}{r}\r)^{\frac{a}{2}-\frac{1}{4}-\frac{b}{2}} \l| J_s\l(\lambda\r)\r| \lesssim 1
\end{equation*}
Proceeding just as in the proof
of off-diagonal estimates in Theorem \ref{ThmMainBounds}, we see:
\begin{equation*}
\begin{split}
\LpN{2}{D_x^\alpha K_{J_s\l(\sqrt{\boxbt}\r)}\l(x,\cdot\r)}{M} &\lesssim \l(1+\frac{s}{r}\r)^{-\frac{a}{2}+\frac{1}{4}+\frac{b}{2}} \LpN{2}{D_x^{\alpha}K_{\l(1+s^2\boxbt\r)^{-\frac{b}{2}}}\l(x,\cdot\r)}{M} \\
&\lesssim \l(1+\frac{\rp{x_0}{y_0}}{r} \r)^{-\frac{a}{2}+\frac{1}{4}+\frac{b}{2}} \frac{\rp{x_0}{y_0}^{-\l|\alpha\r|}}{V\l(x, \rp{x_0}{y_0}\r)}
\end{split}
\end{equation*}
where in the last step, we have applied Proposition \ref{PropOnDiagPolyBound},
using our assumption on $b$.  We have a similar inequality for
$$\LpN{2}{D_y^\beta K_{J_s\l(\sqrt{\boxbt}\r)}\l(\cdot,y\r)}{M}$$
Thus, we have:
\begin{equation*}
\begin{split}
&\l| D_x^{\alpha} D_y^\beta K_{F_s\l(\sqrt{\boxbt} \r)}\l(x,y\r)\bigg|_{\substack{x=x_0\\y=y_0}}\r| \\
&\quad \lesssim \LpN{2}{D_x^{\alpha} K_{J_s\l(\sqrt{\boxbt}\r)}\l(x,\cdot\r)\bigg|_{x=x_0}}{M}
\LpN{2}{D_y^{\beta} K_{J_s\l(\sqrt{\boxbt}\r)}\l(\cdot,y\r)\bigg|_{y=y_0}}{M}\\
&\quad \lesssim \frac{\rp{x_0}{y_0}^{-\l|\alpha\r|-\l|\beta\r|}}{V\l(x_0,\rp{x_0}{y_0}\r)^{\frac{1}{2}}V\l(y_0,\rp{y_0}{y_0}\r)^{\frac{1}{2}}} \l(1+\frac{\rp{x_0}{y_0}}{r}\r)^{-a+\frac{1}{2}+b}
\end{split}
\end{equation*}
Using that $V\l(y_0,\rp{x_0}{y_0}\r)\approx V\l( x_0, \rp{x_0}{y_0}\r)\approx V\l( x_0, r+\rp{x_0}{y_0}\r)$
completes the proof of the bound in the case $\rp{x_0}{y_0}\geq r$.

We now turn to the case when $\rp{x_0}{y_0}\leq r$.  This will
follow in much the same manner as the on-diagonal bounds in Section
\ref{SectionOnDiag}, and we merely sketch the proof.  Let $j\l(\lambda\r)$
be a measurable function such that $j\l(\lambda\r)^2 = m\l( r^2\lambda\r)$.
Note, that by the compact support of $m$, we have for every $N$,
\begin{equation*}
\sup_{\lambda>0} \l(1+r^2\lambda\r)^N \l|j\l(\lambda\r)\r|\lesssim 1
\end{equation*}
with constants independent of $r$.  Here we have used that (by the
Sobolev embedding theorem) $m$ (and therefore $j$) is bounded.
This bound is the point where Remark \ref{RmkUseContVersion}
is essential.
Note also that $m\l(0\r)=0=j\l(0\r)$, and therefore $j\l(\boxb\r)=j\l(\boxbt\r)$.
Thus, we have:
\begin{equation*}
\begin{split}
&\l| D_x^\alpha D_y^\beta K_{m\l(r^2\boxb \r)}\l(x,y\r)\bigg|_{\substack{x=x_0\\y=y_0}}\r|\\
&\quad \leq \LpN{2}{D_x^\alpha K_{j\l(\boxbt\r)}\l(x,\cdot\r)\bigg|_{x=x_0}}{M} \LpN{2}{D_y^\beta K_{j\l(\boxbt\r)}\l(\cdot,y\r)\bigg|_{y=y_0}}{M} \\
&\quad \lesssim \LpN{2}{D_x^\alpha K_{\l(1+r^2\boxbt\r)^{-N}}\l(x,\cdot\r)\bigg|_{x=x_0}}{M}
\LpN{2}{D_y^\beta K_{\l(1+r^2\boxbt\r)^{-N}}\l(\cdot,y\r)\bigg|_{y=y_0}}{M}\\
&\quad\lesssim \frac{r^{-\l|\alpha\r|-\l|\beta\r|}}{V\l(x_0,r\r)^{\frac{1}{2}}V\l(y_0,r\r)^{\frac{1}{2}}} \\
&\quad\lesssim \frac{r^{-\l|\alpha\r|-\l|\beta\r|}}{V\l(x_0,r\r)}
\end{split}
\end{equation*}
Where in the last line, we have used that $V\l(y_0,r\r)\approx V\l(x_0,r\r)$,
since $r\geq \rp{x_0}{y_0}$.  This completes the proof of the proposition, since
$$V\l( x_0,r\r) \approx V\l( x_0, \rp{x_0}{y_0}+r\r)$$
\end{proof}

Proposition \ref{PropMultCptSupp} motivates the following definition:
\begin{defn}\label{DefnElemOps}
Let $r>0$, and $K\in C^{\infty}\l( M\times M\r)$.  We say
$K$ is a pre-$r$-elementary kernel if, for all $N>0$:
\begin{equation*}
\l| D_x^{\alpha} D_y^{\beta} K\l(x,y\r) \r| \leq C_{N,\alpha,\beta}
\l(1+ \frac{\rp{x}{y}}{r}\r)^{-N} \frac{r^{-\l|\alpha\r|-\l|\beta\r|}}{V\l( x, \rp{x}{y} + r\r)}
\end{equation*}
If $S\subset C^\infty\l(M\times M\r)\times \l(0,\infty\r)$ is a set
of pairs $\l( K,r\r)$ where $K$ is a pre-$r$-elementary kernel, we say
the $K$s are uniformly pre-$r$-elementary kernels if the constants
$C_{N,\alpha,\beta}$ can be chosen independently of $\l( K,r\r)\in S$.
\end{defn}

\begin{rmk}
The ``pre'' in the definition above is put there so as to not
conflict with the similar definition in \cite{StreetAnAlgebra}.
\end{rmk}

Next we prove an analog of Lemma 6.36 of
\cite{FollandSteinHardySpacesOnHomogeneousGroups}:
\begin{prop}\label{PropScalingElem}
Suppose $m\in \mathcal{S}\l( \l[0,\infty\r)\r)$, $m\l(0\r)=0$, and $r>0$, then $K_{m\l(r^2\boxb\r)}$
is a pre-$r$-elementary kernel.  Moreover, as $r$ ranges over $\l(0,\infty\r)$
and $m$ ranges over a bounded subset of $\mathcal{S}\l( \l[0,\infty\r)\r)$,
the $K_{m\l(r^2\boxb\r)}$ are uniformly pre-$r$-elementary kernels.
\end{prop}
\begin{proof}
For $m\in C_0^\infty\l(\l(\frac{1}{4},4\r)\r)$ (which is the only
case we shall use in this paper), the result follows
immediately from Proposition \ref{PropMultCptSupp} (by taking $a>>b>>0$).  For the general
case, a proof similar to the one in Proposition \ref{PropMultCptSupp} works,
merely by keeping track of the decay of $m$ at $\infty$.  Since
we do not use this, we leave the details to the reader.
\end{proof}

\begin{lemma}\label{LemmaSumElemOps}
Fix $N_0\in \Z$.  Suppose for each $j\in \Z$, $j\leq N_0$, $K_j$ is a pre-$2^j$-elementary kernel, uniformly
in $j$.  Then,
$K\l(x,y\r):=\sum_{j\leq N_0} K_j\l( x,y\r)$ converges in $C^{\infty}$ off of the
diagonal of $M\times M$.  Moreover, the function $K$ satisfies
the estimates of part 2 of Definition \ref{DefnNisOps}.
\end{lemma}
\begin{proof}
Fix $x,y\in M\times M$, $x\ne y$.
We will show that the sum:
\begin{equation*}
\sum_{-N_1\leq j \leq N_0} \l|D_x^{\alpha} D_y^{\beta} K_j\l( x,y\r)\r|
\end{equation*}
satisfies the desired estimates uniformly in $N_1$.  The result will
then follow immediately.

Consider, suppressing the $-N_1\leq j\leq N_0$, and writing $a=\l|\alpha\r|+\l|\beta\r|$, and $\delta=\rp{x}{y}$,
\begin{equation}\label{EqnSumElemOps1}
\begin{split}
\sum_j \l|D_x^{\alpha} D_y^{\beta} K_j\l( x,y\r)\r| \lesssim \sum_j \l( 1+ \frac{\delta}{2^j} \r)^{-N} \frac{2^{-ja}}{V\l( x, \delta+2^j\r)}
\end{split}
\end{equation}
We separate (\ref{EqnSumElemOps1}) into three sums.  For the first we take $N=0$ and recall the numbers $\delta_0$ and $q$ from Proposition \ref{PropHomogType}:
\begin{equation*}
\begin{split}
\sum_{\delta_0\geq 2^j\geq \delta} \frac{2^{-ja}}{V\l( x, \delta+2^j\r)} &\lesssim \sum_{\delta_0\geq 2^j\geq \delta}\frac{2^{-ja}}{V\l(x,2^j\r)}\\
&\lesssim \sum_{\delta_0\geq 2^j \geq \delta} \l(\frac{\delta}{2^j}\r)^q\frac{2^{-ja}}{V\l(x, \delta\r)}\\
&\lesssim \frac{\delta^{-a}}{V\l(x,\delta\r)}
\end{split}
\end{equation*}
which is the desired bound.  Turning to the second sum:
\begin{equation*}
\begin{split}
\sum_{2^j\leq \delta} \l(1+\frac{\delta}{2^j}\r)^{-N}\frac{2^{-ja}}{V\l( x, \delta+2^j\r)} &\lesssim \sum_{2^j\leq \delta} \l(\frac{2^j}{\delta}\r)^N \frac{2^{-ja}}{V\l(x,\delta\r)}\\
&\lesssim \frac{\delta^{-a}}{V\l(x,\delta\r)}
\end{split}
\end{equation*}
for $N$ sufficiently large.

Finally, the term
$$\sum_{\delta_0\leq 2^j\leq 2^{N_0}} K_j\l( x, y\r)$$
is just a finite sum of $C^\infty$ functions and so satisfies the desired
bounds trivially.
\end{proof}

\begin{proof}[Proof of Theorem \ref{ThmMultipliers}]
Since $m\l( \boxb\r) = m\l( \boxbt\r) + m\l(0\r) \pi$, and
$\pi$ is an NIS operator of order $0$, we need only verify
that $m\l( \boxbt\r)$ is an NIS operator of order $0$.
Property 3 follows just as in the case of $e^{-t\boxbt}$.
The main point is property 2.  Let $\phi\in C_0^\infty\l(\R\r)$
be such that
\begin{equation*}
\begin{split}
\phi\l( x\r) =
\begin{cases}
0 & \text{ if $\l|x\r|\geq 2$,}\\
1 & \text{ if $\l|x\r|\leq 1$}
\end{cases}
\end{split}
\end{equation*}
let $\psi = \phi\l(x\r) - \phi\l( 2x\r)$, and let $m_j\l(\lambda\r) = \psi\l( 2^j \lambda\r) m\l( \lambda\r)$.
Thus, we have that, for $\lambda\ne 0$, $m\l( \lambda\r) = \sum_{j\in \Z} m_j\l( \lambda\r)$, 
and consequently, $\sum_{j\in \Z} m_j\l( \boxb\r) = m\l( \boxbt\r)$, 
with convergence in the strong operator topology.  Note, that since $0$ is an
 isolated point of the spectrum
of $\boxb$ (this follows from the easily provable fact that $\boxbt^{-1}$ is compact), we have that there exists an $N_0$ such that $m_j\l( \boxb\r) =0$
for $j> N_0$.
By Proposition \ref{PropScalingElem}, we have that $m_j\l( \boxb\r)$
is a pre-$2^j$-elementary operator, uniformly in $j$.
Lemma \ref{LemmaSumElemOps} then shows that
$$m\l(\boxbt\r) = \sum_{j\leq N_0} m_j\l( \boxb\r)$$
satisfies property 2 of the definition of NIS operators.  Property 1
follows from Remark \ref{RmkNisDefnsEquiv} and the fact that
$$\sum_{-N\leq j\leq N_0} K_{m_j\l( \boxb\r)}\in C^\infty\l( M\times M\r)$$
for every $N$.  Finally, property 4 follows immediately from what
we have already done.
\end{proof}

We now turn to constructing a Littlewood-Paley square function
which will help us prove Theorem \ref{ThmMultipliersLessSmooth}.
Define $\phi,\psi\in C_0^\infty\l(\R\r)$ as in the proof of Theorem \ref{ThmMultipliers}.
That is $\phi\l(x\r)=1$ if $\l|x\r|\leq 1$, and $\phi\l( x\r)=0$ if $\l|x\r|\geq 2$, 
and $\psi\l( x\r) = \phi\l( x\r) -\phi\l( 2x\r)$; furthermore, we assume
that $\psi$ is real and even.  
Let $\psi_j\l( \lambda\r) = \psi\l( 2^j \lambda\r)$, so that
$\sum_j \psi_j\l(\lambda\r) =1$ for $\lambda\ne 0$.  Define:
$$\psit\l(\lambda\r) = \frac{\psi\l(\lambda\r)}{\sum_{j\in \Z} \l| \psi_j\l( \lambda\r) \r|^2}$$
Thus, if $\psit_j\l(\lambda\r) = \psit\l( 2^j \lambda\r)$, we have:
\begin{equation}\label{EqnSumPsij}
\sum_{j} \psit_j\l( \boxb\r) \psi_j\l(\boxb\r) =1-\pi
\end{equation}

Hence, for $f\in C^\infty\l( M\r)$, if we define
\begin{equation*}
\Lambda\l( f\r) = \l(\sum_j \l| \psi_j\l(\boxb\r)f \r|^2\r)^{\frac{1}{2}},\qquad
\Lambdat\l( f\r) = \l(\sum_j \l| \psit_j\l(\boxb\r)f \r|^2\r)^{\frac{1}{2}}
\end{equation*}
We have, for all $1<p<\infty$:
\begin{equation}\label{EqnSquareLpBound}
\LpN{p}{\Lambda\l(f\r)}{M}\lesssim \LpN{p}{f}{M}, \qquad \LpN{p}{\Lambdat\l(f\r)}{M}\lesssim \LpN{p}{f}{M}
\end{equation}
(\ref{EqnSquareLpBound}) follows from standard arguments.  Indeed, for
any sequence $\epsilon_j$ ($j\in \Z$) of $-1$s and $1$s, we have that:
$$\sum_{j} \epsilon_j \psi_j\l( \boxb\r)$$
is bounded on $L^p$, since it is equal to an NIS operator of order $0$,
just as in the proof of Theorem \ref{ThmMultipliers}.  The result now
follows from the standard trick of taking the $\epsilon_j$s to be
iid random variables of mean $0$ taking values of $\pm 1$.  See
Chapter 4, Section 5 of \cite{SteinSingularIntegralsAndDifferentiablilityPropertiesOfFunctions} and page 267 of \cite{SteinHarmonicAnalysis}.

It now follows, again from standard arguments (\cite{SteinSingularIntegralsAndDifferentiablilityPropertiesOfFunctions,SteinHarmonicAnalysis}), that for $1<p<\infty$:
\begin{equation}\label{EqnSquareLpApprox}
\LpN{p}{\Lambda\l(f\r)}{M} + \LpN{p}{\pi f}{M} \approx \LpN{p}{f}{M}
\end{equation}
for $\lambda\ne 0$.  To see this, it suffices to see that for $f$
such that $\pi f=0$,
\begin{equation*}
\LpN{p}{\Lambda\l(f\r)}{M}\approx \LpN{p}{f}{M}
\end{equation*}
and this follows just as in Chapter 4, Section 5.3.1 of \cite{SteinSingularIntegralsAndDifferentiablilityPropertiesOfFunctions}
by using (\ref{EqnSumPsij}).

Define the maximal function:
\begin{equation*}
\sM\l( f\r)\l( x\r) = \sup_{\delta>0} \frac{1}{V\l(x,\delta\r)} \int_{B\l(x,\delta\r)} \l|f\l(y\r)\r| dy
\end{equation*}
We have:
\begin{lemma}\label{LemmaBoundElemByMax}
Suppose $a>\frac{Q+1}{2}$, $m$ is supported in $\l[ \frac{1}{4},4 \r]$, $r>0$
is fixed, and $\LtaN{a}{m} <\infty$.  Then, there exists a $C=C\l( a,\LtaN{a}{m}\r)$, but not depending on $r$, such that:
\begin{equation*}
\l| m\l(r^2 \boxb\r) f\l( x\r)\r| \leq C \sM\l( f\r)\l(x\r)
\end{equation*}
\end{lemma}
\begin{proof}
Fix $b$ such that $\frac{Q}{2}<b<a-\frac{1}{2}$, and define $\epsilon=a-\frac{1}{2}-b>0$.
Applying Proposition \ref{PropMultCptSupp}, we have that:
\begin{equation*}
\l|K_{m\l(r^2\boxb\r)}\l(x,y\r)\r|\lesssim \l(1+\frac{\rp{x}{y}}{r}\r)^{-\epsilon} \frac{1}{V\l(x,r+\rp{x}{y}\r)}
\end{equation*}
Hence,
\begin{equation*}
\begin{split}
\l|m\l(r^2 \boxb\r)f \l(x\r)\r| &\lesssim \sum_{2^j\geq r} \int_{\rp{x}{y}\leq 2^j} \l(\frac{2^j}{r}\r)^{-\epsilon} \frac{1}{V\l(x,2^j\r)} \l|f\l( y\r)\r|\: dy \\
&\lesssim \sum_{2^j\geq r} \l(\frac{2^j}{r}\r)^{-\epsilon}\sM\l( f\r)\l( x\r)\\
&\lesssim \sM\l(f\r)\l(x\r)
\end{split}
\end{equation*}
\end{proof}

\begin{proof}[Proof of Theorem \ref{ThmMultipliersLessSmooth}]
Take $m$ as in the statement of Theorem \ref{ThmMultipliersLessSmooth}.
We know that $m\l(0\r) \pi$ is bounded on $L^p$, and so it suffices
to show that $m\l( \boxbt\r)$ is bounded on $L^p$.  
The proof will follow from a standard Littlewood-Paley
decomposition, which we sketch.  Fix $f\in C^\infty\l(M\r)$,
and define $F_j=\psi_j\l( \boxb\r) m\l( \boxbt\r) f$.  Note that,
since $\psi_j\l( \boxb\r) \psit_k\l( \boxb\r) =0$ unless $\l|j-k\r|\leq 1$,
we have (applying Lemma \ref{LemmaBoundElemByMax}):
\begin{equation*}
\begin{split}
\l| F_j\r| &= \l| \sum_{k=-1}^1 \psi_j\l(\boxb\r)m\l(\boxbt\r) \psit_{j+k}\l(\boxb\r)\psi_{j+k}\l(\boxb\r)f\r|\\
&\lesssim \sum_{k=-1}^1 \sM\l( \psi_{j+k}\l(\boxb\r) f\r)\\
&\lesssim \l(\sum_{k=-1}^1 \l(\sM\l(\psi_{j+k}\l(\boxb\r) f\r)^2\r)\r)^{\frac{1}{2}}
\end{split}
\end{equation*}
And thus, since $\pi m\l(\boxbt\r)=0$, we have:
\begin{equation*}
\begin{split}
\LpN{p}{m\l(\boxbt\r) f}{M} &\approx \LpN{p}{\l(\sum_j\l|F_j\r|^2\r)^{\frac{1}{2}}}{M}\\
&\lesssim \LpN{p}{\l(\sum_j \sum_{k=-1}^1 \sM\l(\psi_{j+k}\l(\boxb\r)f\r)^2\r)^{\frac{1}{2}}}{M}\\
&\lesssim \LpN{p}{\l(\sum_j \sM\l(\psi_{j}\l(\boxb\r)f\r)^2\r)^{\frac{1}{2}}}{M}\\
&\lesssim \LpN{p}{\Lambda\l(f\r)}{M}\\
&\lesssim \LpN{p}{f}{M}
\end{split}
\end{equation*}
where we have used the vector-valued inequality for $\sM$.  See
Chapter 2, Section 1 of \cite{SteinHarmonicAnalysis}.
\end{proof}

\section{Other Examples}\label{SectionOtherExamples}
	In this section we state some results with $\boxb$ replaced by
other operators.  In each case $\pi$ will denote the
orthogonal projection onto the $L^2$ kernel of the operator in
question.  All of the proofs of the below results are similar
to the proofs above, and we therefore confine ourselves to
brief comments on the necessary changes.

	\subsection{Example: A Generalization of Theorem 4 of \cite{SikoraRieszTransformGaussianBoundsAndTheMethodOfWaveEquation}}\label{ExampleSik}
		In this example, we discuss how the above methods may be applied
in the general situation of \cite{SikoraRieszTransformGaussianBoundsAndTheMethodOfWaveEquation}, except that we allow the infinitesimal generator
of the heat semi-group to have \it non-trivial \rm $L^2$ kernel.

We first quickly review the setup of that paper, though we refer
the reader there for more rigorous details.  Let $X$ be a 
metric measurable space with metric $\rho$, and let $\mu$ be a Borel measure on $X$.
We define $B\l( x, \delta\r)$ and $V\l( x,\delta\r)$ as above, but
with $\mu$ in place of Lebesgue measure.  We suppose that:
$$V\l( x, \gamma \delta\r) \lesssim \gamma^Q V\l( x, \delta\r)$$
for all $\gamma\geq 1$.

We suppose $TX$ is a continuous vector bundle with base $X$ (with
fibers $\C^d$), and scalar product $\l( \cdot, \cdot\r)_x$.  We
define the space $L^2\l( TX,\mu\r)$ of \it sections \rm of $TX$
in the usual way.  Now, suppose that $\sL$ is a, possibly unbounded, self-adjoint
positive semi-definite operator acting on $L^2\l( TX, \mu\r)$.
For bounded Borel measurable functions $m$, we define
$m\l( \sL\r)$ and $m\l( \sLt\r)$ analogous to the definitions
earlier in the paper for $\boxb$.  Let $\pi$ be the orthogonal projection
onto the $L^2$ kernel of $\sL$. 

Suppose that, for bounded continuous sections $\phi, \psi\in L^2\l( TX,\mu\r)$ with disjoint support, we have that:
$$\l< \psi, \pi\phi \r>=\int \l(\psi\l(x\r), K_\pi \l(x,y\r) \phi\l(y\r)\r)_x dy dx $$
for a measurable function $K_\pi$ defined on $X\times X$ without
the diagonal, and taking values in $Hom\l( T_yX, T_xX\r)$.
We suppose that, for $x\ne y$:
$$\l|K_\pi\l( x,y\r)\r|\lesssim \frac{1}{V\l(x,\rp{x}{y}\r)}$$
where $\l| \cdot\r|$ denotes the operator norm on $Hom\l( T_yX, T_xX\r)$.

We suppose that $e^{-t\sLt}$ satisfies the on-diagonal estimate:
\begin{equation*}
\LpN{2}{\l| K_{e^{-t\sLt}}\l(x,\cdot\r) \r|}{M} \lesssim V\l(x, \sqrt{t}\r)^{-\frac{1}{2}}
\end{equation*}
and that $\cos\l(t\sqrt{\sL}\r)$ has finite propagation
speed, in the sense used in \cite{SikoraRieszTransformGaussianBoundsAndTheMethodOfWaveEquation}.  Informally, that:
$$\supp\l(K_{\cos\l(t\sqrt{\sL}\r)}\r)\subseteq \l\{\l(x,y\r): \rp{x}{y}\leq t \r\}$$

Then, $e^{-t\sL}$ satisfies the off diagonal estimates, for $t<\rp{x}{y}^2$,
\begin{equation*}
\l| K_{e^{-t\sL}}\l(x,y\r) \r|\lesssim V\l( x, \rp{x}{y}\r)^{-1} \l(\frac{\rp{x}{y}^2}{t}\r)^{Q-\frac{1}{2}} e^{-\frac{\rp{x}{y}^2}{4t}}
\end{equation*}
The proof of this fact follows just by putting together the methods
of \cite{SikoraRieszTransformGaussianBoundsAndTheMethodOfWaveEquation}
and the methods of this paper.
\begin{rmk}
Actually, \cite{SikoraRieszTransformGaussianBoundsAndTheMethodOfWaveEquation}.
has a slightly better bound.  This is due to the fact that there
is a slight difficulty meshing the bounds for the heat kernel
with those for $\pi$.
\end{rmk}

	\subsection{Example: Pseudoconvex CR Manifolds of Finite Type}\label{ExamplePseudoConvex}
		In this example, we let $M$ be a compact pseudoconvex CR manifold
of dimension $2n-1$ ($n\geq 3$), and we assume that the range of $\dbarb$ (as an operator on $L^2\l(M\r)$) is closed, and we assume that $M$ is of
finite commutator type.
Let $x_0\in M$ be a fixed base point, and let $U$ be a neighborhood of $x_0$.
We think of $U$ as small and may shrink it throughout the discussion.
Fix a local basis $L_1,\ldots,L_{n-1}$ for $T^{1,0}$ on $U$ (which
we may do by making $U$ small enough).  Fix a Hermitian metric
on $\C TM$ such that $L_1,\ldots, L_{n-1}$ are orthonormal.

Put $X_j= \Re\l(L_j\r), X_{j+n-1}=\Im\l( L_j\r)$, by assumption
the $X_k$s along with their commutators up to a certain fixed
order span the tangent space $TU$; we use these vector fields
to define a metric $\rho$ as in Section \ref{SectionGeomOfM}
and define $D^\alpha$ for an ordered multi-index $\alpha$ in terms
of these vector fields, as well.  In addition, we assume that
condition $D\l( q\r)$ holds on $U$.
This is the setup of \cite{KoenigOnMaixmalSobolevAndHolderEstimates},
and we refer the reader there for more details.

We use a definition from \cite{KoenigOnMaixmalSobolevAndHolderEstimates}:
\begin{defn}
An operator $T$ on functions $f\in C^\infty \l(M\r)$ is said to be an NIS
operator smoothing of order $r$ in $U$ if $T$ satisfies the properties
of Definition \ref{DefnNisOps} except for the following modifications:
\begin{itemize}
\item In property 2, we only consider $x,y\in U$.
\item In property 3, we only consider $x$ and $\delta$ such that
$B\l(x,\delta\r)\subset U$.
\end{itemize}
This definition extends to operators on forms in the obvious way;
see \cite{KoenigOnMaixmalSobolevAndHolderEstimates}, page 158.
\end{defn}

Consider the operator $\dbarb$ acting on $\l(0,q\r)$ forms.  We define
the operator $\sL = \dbarb^{*}\dbarb$ acting on $\l(0,q\r)$ forms
via the Hermitian product that we fixed above (here $\dbarb^{*}$ is
acting on $\l(0,q+1\r)$ forms).  Let $\pi$ be the projection onto
the $L^2$ kernel of $\sL$.  It is shown in \cite{KoenigOnMaixmalSobolevAndHolderEstimates} that $\pi$ is an NIS operator of order $0$ in $U$
and that the relative fundamental solution $\sLt^{-1}$ is an NIS
operator of order $2$ in $U$.

One may write $\dbarb$ and $\dbarb^{*}$ in terms of $\overline{L_j}$ (see
(2.6) and (2.7) of \cite{KoenigOnMaixmalSobolevAndHolderEstimates})
and with this a proof almost exactly the same as the one above for Theorem \ref{ThmFiniteProp}
shows that
$$\supp \l(K_{\cos\l(t\sqrt{\sL}\r)}\r)\cap U\times U \subseteq\l\{ \l(x,y\r)\in U\times U: \rp{x}{y}\leq \kappa t \r\}$$
for some fixed constant $\kappa$.  The results in Section \ref{SectionResults}
hold with the following modifications:
\begin{itemize}
\item The bounds in Theorem \ref{ThmMainBounds} and Corollary \ref{CorCleanerMainBound} hold for $x,y\in U$.
\item In Theorem \ref{ThmHeatOpsAreNIS} and Theorem \ref{ThmMultipliers}
``NIS operators'' must be replaced with ``NIS operators in $U$.''
\item In Theorem \ref{ThmMultipliersLessSmooth}, we consider $m\l( \sL\r)$
taking $L^p\l(U\r)\rightarrow L^p\l(U\r)$.
\end{itemize}
The proofs are essentially the same as the ones in this paper, however
one must work with operators on forms as in \cite{SikoraRieszTransformGaussianBoundsAndTheMethodOfWaveEquation} and Example \ref{ExampleSik}.  We 
leave the details to the reader.

	\subsection{Example: Polynomial Model Domains}\label{ExamplePolyModel}
		In this example, we discuss the other case treated in
\cite{NagelSteinTheBoxbHeatEquationOnPseudoconvexManifoldsOfFiniteType}.
In this case, there is a subharmonic, nonharmonic polynomial $h:\C\rightarrow\R$
such that
$$M=\l\{ \l( z,w\r)\in \C^2: \Im\l( w\r)=h\l(z\r) \r\}$$
We define $\boxb$ as in that reference, and refer the reader
there for details.  All of the results in Section \ref{SectionResults}
hold without any changes in the statement of the results, however,
we must address a few differences in the proofs.

The analog of Theorem \ref{ThmFiniteProp} follows just as before,
just by using the smooth metric constructed for this case
in Section 4 of \cite{NagelSteinDifferentiableControlMetricsAndScaledBumpFunctions}.  

The main differences between this case and the case treated above are (the closely related facts)
that $\boxbt^{-1}$ is not bounded on $L^2$ and that the spectrum
of $\boxb$ is not discrete.  The fact that $\boxbt^{-1}$ is not
bounded on $L^2$ can be worked around by using that we may take
$R_0=\infty$ in Theorem \ref{ThmScaledEst}.

That the spectrum of $\boxb$ is not discrete forces us to have a
replacement for Lemma \ref{LemmaSumElemOps}.  Indeed, we replace it
with the same result but with $N_0=\infty$.  This is proven
in a similar manner, by using the fact that (in this case) 
we may take $\delta_0$ in Proposition \ref{PropHomogType} to
be $\infty$.
At this point, all of the proofs go through
with only minor changes.  We leave the details to the reader.
For some related results, see \cite{RaichPointwiseEstimatesForRelativeFundamentalSolutionsOfHeatEquations}.


	\subsection{Example: Operators on a Compact Manifold, Defined by Vector Fields}\label{ExampleCptMfld}
		For this example, let $M$ be a compact Riemannian manifold, and let $X_1,\ldots, X_n$ be vector fields satisfying H\"ormander's condition.
Let $\sL$ be an second order, self adjoint, polynomial in the
vector fields $X_1,\ldots,X_n$.  Using these vector fields, we obtain a 
metric $\rho$, as in Section \ref{SectionGeomOfM}.
We assume that $\sL$ has the following properties:
\begin{itemize}
\item The $\sL$ wave operator, $\cos\l( t\sqrt{\sL}\r)$, has finite
propagation speed.  That is, it satisfies the conclusion of
Theorem \ref{ThmFiniteProp}.
\item The relative fundamental solution, $\sLt^{-1}$, of $\sL$
is an NIS operator of order $2$.
\end{itemize}
Then, all of the results in Section \ref{SectionResults} remain true
with $\sL$ in place of $\boxb$, with essentially the same proofs.
\begin{rmk}
Actually, that $\sL$ be of second order is inessential.  We leave
such generalizations to the reader.
\end{rmk}

In particular, all of the proofs in this paper work with $\sL$ equal
to the sublaplacian:
$$\sL=X_1^{*}X_1+\cdots+X_n^{*}X_n$$
where we identify $\sL$ with its Friedrich extension.  
The finite propagation speed follows just as in the proof of
Theorem \ref{ThmFiniteProp}.
In this case,
$\pi$ is just the projection onto the constant functions.
Note that, since $e^{-t\sL}= e^{-t\sLt}+\pi$, bounds for $e^{-t\sL}$
and $e^{-t\sLt}$ are essentially the same.  In the case
of the sublaplacian, though, most of the results in this paper
are quite well known.  For instance, Theorem \ref{ThmFiniteProp}
can be found in \cite{MelrosePropagationForTheWaveGroupOfAPositiveSubellipticSecondOrderDifferentialOperator}, Corollary \ref{CorCleanerMainBound}
can be found in \cite{JerisonSanchezCalleEstimatesForTheHeatKernelForASumeOfSquares}, and Theorem \ref{ThmHeatOpsAreNIS} was implicitly proven in \cite{NagelSteinTheBoxbHeatEquationOnPseudoconvexManifoldsOfFiniteType}.


	\subsection{Example: Quasi-homogeneous Vector Fields}\label{ExampleQuasiHomog}
		In this example, we let $X_1,\ldots, X_n$ be vector fields on $M=\R^d$
satisfying H\"ormander's condition, and which are homogeneous
of degree $1$ with respect to a one parameter family of dilations
on $\R^d$.  An example would be the left invariant vector fields of
degree $1$ on a stratified group.  Let $\sL$ be a second order,
self-adjoint, homogeneous polynomial in $X_1,\ldots,X_n$, and
assume that $\sL$ satisfies the same two assumptions as in
Example \ref{ExampleCptMfld}.  Then, all of the results in
Section \ref{SectionResults} go through with proofs almost
exactly the same as Example \ref{ExamplePolyModel}.  

Just as in Example \ref{ExampleCptMfld}, the sublaplacian:
$$\sL=-X_1^2-\cdots-X_n^2$$
is a special case of this.  The finite propagation
of the wave equation may be verified for $x,y$ in a fixed compact
neighborhood of $0$ by the same proof as in Theorem \ref{ThmFiniteProp}, and
then extended to all $x,y$ by homogeneity.  
In this case, $\pi=0$, and so
$e^{-t\sL}$ and $e^{-t\sLt}$ satisfy the same bounds.  Just
as in Example \ref{ExampleCptMfld}, these results are well known.
In particular, in the case of the sublaplacian on a stratified group,
all of the results in this paper can be improved, and are quite well known.  See
\cite{FollandSteinHardySpacesOnHomogeneousGroups,ChristLpBoundsForSpectralMultipliersOnNilpotentGroups,AlexopoulosSpectralMultipliersOnLigeGroupsOfPolynomialGrowth}, and references therein.

\bibliographystyle{amsalpha}

\bibliography{heat}

\end{document}